\documentclass[12pt]{amsart}

\usepackage[T1]{fontenc}
\usepackage[utf8]{inputenc}
\usepackage{lmodern}
\usepackage{amsmath,amssymb,amsthm,mathtools}
\usepackage{mathrsfs}
\usepackage{xcolor}
\usepackage{tikz-cd}
\usepackage{enumitem}
\usepackage{geometry}
\usepackage{hyperref}
\usepackage{setspace}
\usepackage[numbers,sort&compress]{natbib}
\hypersetup{colorlinks=true, linkcolor=blue, citecolor=blue, urlcolor=blue}

\newcommand{\R}{\mathbb R}

\newcommand{\N}{\mathbb N}
\newcommand{\Z}{\mathbb Z}
\newcommand{\eps}{\varepsilon}
\newcommand{\Id}{\operatorname{Id}}
\newcommand{\vol}{\operatorname{vol}}
\newcommand{\diam}{\operatorname{diam}}
\newcommand{\Ric}{\operatorname{Ric}}
\newcommand{\Rm}{\operatorname{Rm}}

\newcommand{\reg}{\operatorname{reg}}

\newcommand{\Spec}{\operatorname{Spec}}
\newcommand{\Spin}{\operatorname{Spin}}
\newcommand{\SO}{\operatorname{SO}}

\newcommand{\cyl}{\operatorname{cyl}}
\newcommand{\ALE}{\operatorname{ALE}}

\theoremstyle{plain}
\newtheorem{theorem}{Theorem}[section]
\newtheorem{proposition}[theorem]{Proposition}
\newtheorem{lemma}[theorem]{Lemma}

\theoremstyle{remark}
\newtheorem{remark}[theorem]{Remark}

\numberwithin{equation}{section}

\title[Energy quantization for Dirac systems]{Energy quantization for Dirac systems over non-collapsed degenerating Einstein manifolds}

\author{Pan Chen}
\address{School of Mathematical Sciences, Shanghai Jiao Tong University\\ 800 Dongchuan Road \\ Shanghai, 200240 \\P. R. China}%
\email{chenpan\_sj@sjtu.edu.cn}%

\author{Youmin Chen}
\address{Department of Mathematics, Shantou University\\ 5 Cuifeng Road \\ Shantou, Guangdong, 515063 \\P. R. China}%
\email{youminchen@stu.edu.cn}%

\author{Miaomiao Zhu}
\address{School of Mathematical Sciences, Shanghai Jiao Tong University\\ 800 Dongchuan Road \\ Shanghai, 200240 \\P. R. China}%
\email{mizhu@sjtu.edu.cn}%
\date{\today}

\subjclass[2020]{53C25, 35J46, 35B44}
\keywords{Einstein manifolds; degenerating; Dirac system; energy identity.}

\begin{document}

\begin{abstract}
We study energy quantization for a class of Dirac systems on compact spin
Einstein manifolds of dimension \(n\). For a sequence of solutions to a
nonlinear Dirac system with uniformly bounded energy on a fixed spin Riemannian manifold, we first establish an energy identity theorem.
We then investigate the more complicated case of underlying domain manifolds being a sequence of non-collapsed degenerating spin Einstein manifolds. At an orbifold singular point, three types of bubble spinors can possibly appear, living respectively on \(\mathbb{R}^n\), on a Ricci-flat ALE bubble space, and on the flat cone \(\mathbb{R}^n/\Gamma\). By developing asymptotic analysis for solutions over degenerating neck regions, we establish that energy identity holds.
\end{abstract}

\maketitle
%\tableofcontents

\section{Introduction}

\vskip0.5cm

Let $M$ be an $n$-dimensional compact spin manifold equipped with a Riemannian metric
 $g$ and a spin structure $\Lambda : P_{\Spin}(M) \to P_{\SO}(M)$.
  Denote by $\mathbb{S}(M)$ the spinor bundle on $M$ and by $D_g$ the
   Dirac operator. For basic facts concerning spin structures and Dirac operators
    on manifolds, we refer the reader to Section \ref{section2}.
In this article, we mainly consider the compactness
problem for solutions of the following nonlinear Dirac system on $M$:
\begin{equation}\label{eq:intro-system}
 {D_g\psi^i=\sum_{j,l,t=1}^d H^i_{jlt}(x)\,
 \mathcal{N}_n\left( \langle \psi^j,\psi^l\rangle\right)
 \psi^t,
        \qquad i=1,\ldots,d,}
\end{equation}
where $\psi=(\psi^1,\cdots,\psi^d)\in (\Gamma({\mathbb S(M)}))^d$,
 {$H^i_{jlt}\in C^1(M, \mathbb R)$,}  
\[
\mathcal N_n(z)=
\begin{cases}
|z|^{\frac{1}{n-1}}, & n\geq 3,\\
\Theta(z), & n=2,
\end{cases}
\]
and in dimension two, \(\Theta\) denotes either of the two choices
\[
\Theta(z)=|z|
\qquad\text{or}\qquad
\Theta(z)=z.
\]

In dimension two, nonlinear Dirac equations arise naturally from
 the surface theory.
When $d=1$, we briefly recall how \eqref{eq:intro-system} appears
from the spinorial Weierstrass representation, see \cite{Kenmotsu1979,Friedrich1998,Taimanov2006}.
 Let
\[f=(x^1,x^2,x^3):M\to\mathbb R^3\] be a conformal immersion of
 a Riemann surface and let \(z=x+iy\) be a local conformal
 coordinate.  Since \((f_z,f_z)=0\), the vector
  \(f_z\in\mathbb C^3\) lies on the quadric
   \(y_1^2+y_2^2+y_3^2=0\).  A rational parametrization
   of this quadric introduces a spinor \(\psi=(\psi_1,\psi_2)\)
   by
\[
  x^1_z=\frac{i}{2}\left(\overline{\psi_2}^{\,2}+\psi_1^2\right),\qquad
  x^2_z=\frac{1}{2}\left(\overline{\psi_2}^{\,2}-\psi_1^2\right),\qquad
  x^3_z=\psi_1\overline{\psi_2}.
\]
The corresponding real coordinate one-forms are $(\omega^1, \omega^2, \omega^3)$.
On a simply connected domain the immersion is recovered from
\[f(p)=f(p_0)+\int_{p_0}^p(\omega^1,\omega^2,\omega^3).\]
The closedness of one-forms $\omega^k$ are equivalently written as the existence of a real potential \(U\) such that
\begin{equation}\label{linear-eq}
         D_{\mathbb{R}^2}\psi+ U\psi=0
\end{equation}
and conversely this Dirac equation makes the forms \(\omega^k\) closed.
The potential \(U\) is determined by the geometry of the surface.  If
\[
  ds^2=e^{2\alpha}dz\,d\bar z,
  \qquad
  e^\alpha=|\psi_1|^2+|\psi_2|^2,
\]
then
\[
  U=\frac{H e^\alpha}{2},
\]
where \(H\) is the mean curvature.  Therefore \eqref{linear-eq} becomes
\[
  D_{\mathbb{R}^2}\psi=-\frac{H}{2}|\psi|^2\psi.
\]
 For \(H=0\) one obtains the free Dirac equation, which is the spinorial form of the classical Weierstrass representation for minimal surfaces.
The same origin can be described intrinsically.  For an isometric immersion
\[f:(M,\tilde g)\to\mathbb R^3,\] the restriction of a parallel spinor in \(\mathbb R^3\) gives a constant length spinor \(\varphi\) on \((M,\tilde g)\) satisfying
\begin{equation}\label{linear-eq2}
        D_{\tilde g}\varphi=H\varphi,
  \qquad |\varphi|_{\tilde g}=1 .
\end{equation}
where $H$ is the mean curvature. If \(\tilde g=e^{2u}g\) and \(\psi=e^{u/2}\varphi\), the conformal covariance
\[
  D_{\tilde g}\varphi=e^{-3u/2}D_g(e^{u/2}\varphi),
  \qquad |\psi|_g^2=e^u,
\]
turns \eqref{linear-eq2} into
\[
  D_g\psi=H|\psi|_g^2\psi.
\]

Another geometric origin of \eqref{eq:intro-system} in dimension two comes
from Dirac-harmonic maps with a curvature term, which were first proposed and studied in \cite{ChenJostWang2007}. Let $(M,g)$ be a Riemann
surface with a fixed spin structure and let $\phi$ be a smooth
map from $M$ to a Riemannian manifold $(N , h)$
of dimension $d\geq 2$. In local coordinates $\{x_\alpha\}$
and $\{y^i\}$ on $M$ and $N$ respectively,
a section of the twisted
bundle $\mathbb{S}(M)\otimes \phi^{-1}TN$, can be written as
\[
    \psi=\psi^i\otimes \partial_{y^i}(\phi).
\]
The connection $\widetilde{\nabla} $ on $\mathbb{S}(M)\otimes \phi^{-1}TN$ is induced
by the spin connection on $\mathbb{S}(M)$ and the Levi-Civita
 connection of $N$. The corresponding Dirac operator
 along the map $\phi$ is given locally by
\[
\widetilde{D}_M\psi=
D_M \psi^i\otimes \partial_{y^i}(\phi)
+
\Gamma^i_{jk}(\phi)\nabla_{e_\alpha}\phi^j, e_\alpha\cdot \psi^k\otimes \partial_{y^i}(\phi),
\]
where $D_M$ denotes the usual Dirac operator on $M$ and $\Gamma_{jk}^i$ stands for the Christoffel
symbols of $N$.
Dirac-harmonic maps with a curvature term are critical points of the functional
\[
    L_c(\phi,\psi)
    =
    \frac12\int_M
    \left(
        |d\phi|^2+\langle \psi,\widetilde{D}_M\psi\rangle
        -\frac16 R_{ikjl}
        \langle \psi^i,\psi^j\rangle
        \langle \psi^k,\psi^l\rangle
    \right)dV_g ,
\]
This functional is dictated by the nonlinear
   supersymmetric $\sigma$ model in quantum field theory, one can refer to \cite{Deligne1999}.
 The spinor part of
its Euler--Lagrange equations is
\[
   \widetilde{D}_M\psi^i
    =
    -\frac13 R^i_{jlt}(\phi)
    \langle \psi^j,\psi^l\rangle \psi^t ,
    \qquad i=1,\ldots,d .
\]
where $R^i_{jlt}$ stands for a component of the curvature tensor of
$N$. In particular, if $\phi$ is a constant map, then the above equation reduces to
\[
    D_M \psi^i
    =
    -\frac13 R^i_{jlt}
    \langle \psi^j,\psi^l\rangle \psi^t ,
    \qquad i=1,\ldots,d .
\]
which is a Dirac equation of type \eqref{eq:intro-system}.

Physically, nonlinear Dirac equations \eqref{eq:intro-system} also arise in models motivated by
  quantum field theory, such as the Soler, Gross-Neveu, Thirring and
   related model, see \cite{Soler1970,Thirring1958,GrossNeveu1974} for more details.
Furthermore, in general dimensions, a classical example of equation \eqref{eq:intro-system}
comes from the spinorial Yamabe problem,
see, for example, \cite{Bartsch2021,Sire2023,Isobe2023,Isobe2024,Maalaoui2026,Ammann2003} and references therein.

We denote the energy of $\psi\in L^{\frac{2n}{n-1}}(M,{\mathbb S(M)})$ as
\[
        E(\psi)=\int_M |\psi|_g^{\frac{2n}{n-1}}\,dV_g .
\]
This functional is conformally invariant; namely, for any $u\in C^\infty(M)$ and a conformal
metric $g_1=e^{2u}g$, there exists a map between spinor bundles
$\beta_{g,g_1}:{\mathbb S(M,g)}\to {\mathbb S(M,g_1)}$ such that
\[
        \int_M |\psi|_g^{\frac{2n}{n-1}}\,dV_g
        =\int_M |\phi|_{g_1}^{\frac{2n}{n-1}}\,dV_{g_1},
        \qquad
        \phi=\beta_{g,g_1}\big(e^{-\frac{n-1}{2}u}\psi\big).
\]
Consequently, due to the possible concentration of energy at isolated points of the domain,
a sequence of solutions to \eqref{eq:intro-system} with uniformly bounded energy generally
does not possess strong compactness. To study the compactness issue, one must rescale around
the energy concentration points to obtain bubble solutions. Such bubbling phenomena are the spinorial counterpart
of the well-known lack of compactness in the Yamabe problem and other geometric variational problems. When the domain is a 2-dimensional Euclidean space, the blow-up theory for
\eqref{eq:intro-system} has been established through various methods. More precisely, for a
sequence of solutions with uniformly bounded energy, an energy identity holds, which yields
the compactness of the solution sequence modulo bubbles, see \cite{Zhu2016,Chen2024,ChenJostWang2008}.
However, for higher dimensions, no such result seems to be available; this is the first problem to be discussed in the present article. We
will first study the compactness, modulo finitely many bubbles, of a sequence of solutions
\eqref{eq:intro-system} with uniformly bounded energy on a fixed $n$-dimensional Spin
manifold. Our main results include the energy identity. The core of the proof lies in developing
a new three circle theorem for the Dirac operator. Our first results are stated as follows.

\begin{theorem}\label{thm:fixed}
Let $n\geq 2$, $\psi_k\in L^{\frac{2n}{n-1}}(M,{\mathbb S(M)})^d$ be a sequence of weak solutions to
\eqref{eq:intro-system} with uniformly bounded energy
\[
        E(\psi_k)=\int_M |\psi_k|_g^{\frac{2n}{n-1}}\,dV_g\le C,
\]
where $C>0$. Assume that $\psi_k$ converges to $\psi$ weakly, and
$p_1,\ldots,p_I$ is the set of blow-up points. Then there exists a subsequence, still denoted by
$\{\psi_k\}$, and finitely many solutions $\{\xi_t^l\}$ on $S^n$ of
\eqref{eq:intro-system} with $H^i_{jlt}=H^i_{jlt}(p_t)$, where
$t=1,\ldots,I$ and $l=1,\ldots,L_t$, such that
\[
        \lim_{k\to\infty}E(\psi_k)
        =E(\psi)+\sum_{t=1}^I\sum_{l=1}^{L_t}E(\xi_t^l).
\]
\end{theorem}
The main focus of this paper is the case in which the domain manifolds
themselves degenerate.  This is a different compactness problem from the
 fixed-domain case. In two dimensions, when domain
 Riemann spin surface converge to
  a possibly noncompact Riemann spin surface with all
punctures (if there are any) of Neveu-Schwarz type, for 1st order
Dirac systems, energy identity holds, see \cite{Chen2024,Zhu2016}.
 However, in many conformally invariant variational problems,
 allowing the domain to vary creates long neck regions where energy may persist
 even after all ordinary bubbles have been extracted.
 An example is a sequence of harmonic maps from degenerating Riemann surfaces,
 where energy identities may fail on long collars unless extra
  conditions are imposed, see \cite{Zhu,Zhu2009,LiuZhu2024}.
  Thus, for geometric PDEs on varying domains,
  one must understand not only the concentration of the unknown
  fields but also the geometric degeneration of the underlying manifolds.

In the present paper the varying domains are non-collapsed Einstein spin manifolds.
 The relevant compactness theory says that,
 under the conditions of a uniform upper bound on diameter and a uniform lower bound on volume, together with an $L^{n/2}$ curvature upper bound,
  a subsequence of $(M_k,g_k)$ converges in
  the Gromov-Hausdorff sense to a compact Einstein orbifold $(M_\infty,g_\infty)$,
  and the convergence is smooth in the Cheeger-Gromov sense
  away from finitely many orbifold points;
   see \cite{Anderson1989,Anderson1992,Bando,Nakajima1988,Nakajima1994,CheegerTian}.
    Near each orbifold point $x_a$, the curvature may concentrate.
    After choosing points $x_{a,k}\to x_a$ and a scale $r_k\to0$,
    the rescaled manifolds $(M_k,r_k^{-2}g_k,x_{a,k})$ converge to a complete,
    noncompact, Ricci-flat, non-flat ALE space $(M_a,h_a)$, possibly with orbifold
    singularities.  Repeating this process at the singular points of the ALE
    limits produces a finite ALE bubble tree. Thus a neighbourhood of such an orbifold
    point is decomposed, at a rough geometric level, into three regions: the regular part of the
orbifold limit, the ALE bubble region, and a long degenerating neck
which is asymptotic to a portion of the flat cone \(\mathbb R^n/\Gamma\).
This geometric decomposition has to be incorporated into the analysis
of the equation itself. The work on biharmonic maps from
non-collapsed degenerating Einstein \(4\)-manifolds \cite{ChenZhu2024} shows that, when
the domain metrics degenerate, the usual neck
analysis is no longer sufficient by itself.  The blow-up of the
solutions and the blow-up of the domain interact through the
neck geometry; in particular, one has to use the fact that there is no
curvature concentration on the degenerating necks and that the metrics
there are sufficiently close to flat cone metrics in suitable
coordinates.  Their analysis provides an important model for treating
geometric PDEs on domains whose geometry degenerates, see \cite{ChenZhu2024} for more details.

More precisely, let
\((M_k,g_k,\Lambda_k)\) be a sequence of smooth, closed
Einstein spin \(n\)-manifolds ($n\geq 4$) satisfying
\[
        \operatorname{Ric}_{g_k}=\lambda_k g_k,
        \qquad \sup_k |\lambda_k|<\mu .
\]
Assume moreover that there exist constants \(D,V,R>0\) such that
\begin{equation}\label{eq:intro-einstein-assumption}
        \operatorname{diam}(M_k,g_k)\le D,\qquad
        \operatorname{vol}(M_k,g_k)\ge V,\qquad
        \int_{M_k}|\operatorname{Rm}_{g_k}|^{\frac n2}\,dV_{g_k}\le R .
\end{equation}
Here \(\operatorname{Rm}_{g}\) denotes the Riemann curvature tensor of the
metric \(g\).
It is worth noting that, in dimension four, the last assumption in
\eqref{eq:intro-einstein-assumption} is automatically satisfied under the
uniform upper bound on the diameter, the uniform lower bound on the volume,
and the uniform bound on the Einstein constants, by the seminal work
\cite{Cheeger2015}. More precisely, there exists a constant \(R>0\),
depending only on these uniform bounds, such that
\[
        \int_{M_k}|\operatorname{Rm}_{g_k}|^{2}\,dV_{g_k}\le R .
\]
Let
\[
        \psi_k=(\psi_k^1,\ldots,\psi_k^d)
        \in \Gamma(\mathbb S(M_k,g_k))^d
\]
be solutions of  Dirac system
\begin{equation}\label{eq:varying-system}
        D_{g_k}\psi_k^i
        =
        \sum_{j,l,t=1}^d H
        |\langle\psi_k^j,\psi_k^l\rangle|^{\frac1{n-1}}\psi_k^t,
        \qquad i=1,\ldots,d,
\end{equation}
with
\begin{equation}\label{eq:intro-energy-bound-varying}
        \int_{M_k}|\psi_k|^{\frac{2n}{n-1}}\,dV_{g_k}\le C .
\end{equation}
At an orbifold singular point \(x_a\) of the limiting
orbifold, there are two possible scales: the geometric scale
\(r_k\), at which the ALE bubble appears, and the analytic scale $\lambda_k$ at
which the critical energy $E(\psi_k)$ concentrates.  The relative position
of these two scales determines the limiting domain of the rescaled
spinors. If the spinor concentration takes place at a scale much smaller than
the relevant geometric scale, the rescaled domains become Euclidean
and one obtains a bubble on \(\mathbb R^n\).  If the concentration
occurs at the ALE scale \(r_k\), the limiting spinor is defined on the
Ricci-flat ALE bubble space $M_a$.  Finally, if the concentration occurs
inside the degenerating neck, the limiting domain is the flat cone
\(\mathbb R^n/\Gamma\).  These three kinds of bubbles will be denoted
below by
\[
        \xi^A:\mathbb R^n\to  \Gamma(\mathbb{S}(\mathbb{R}^n))^d,\qquad
        \xi^B:M_{a}\to \Gamma(\mathbb{S}(M_a))^d,\qquad
        \xi^C:\mathbb R^n/\Gamma\to \Gamma(\mathbb{S}(\mathbb{R}^n/\Gamma))^d,
\]
respectively.  The main point of the degenerating analysis is to show
that, after all such bubbles have been extracted, there is no energy loss on the neck regions.
We state our main result in this paper as follows.

\begin{theorem}\label{thm:main}
Let $n\geq 4$, $(M_k,g_k,\Lambda_k)$ and $\psi_k$ be as above. Suppose that there is only one ALE
bubble space at the orbifold point $x_a$ and that there is at most one bubble spinor in each
case of the blow-up scheme discussed below. Then, up to a subsequence, we have the energy
identity
\begin{align*}
 \lim_{k\to\infty}\int_{B_{g_k}(x_{a,k},\delta_0)} |\psi_k|^{\frac{2n}{n-1}}\,dV_{g_k}
 =\int_{B_{g_\infty}(x_a,\delta_0)}
        |\psi_\infty|^{\frac{2n}{n-1}}\,dV_{g_\infty}
 +\sum_{a=A,B,C}E(\xi^a),
\end{align*}
where $\xi^A$ is a bubble spinor on $\R^n$, $\xi^B$ is a bubble
spinor on the ALE bubble space, $\xi^C$ is a bubble spinor on the flat cone
$\R^n/\Gamma$.
\end{theorem}
The general case of multiple bubble manifolds and multiple bubble spinors follows from induction
arguments.

We now briefly describe the organization of the paper.
In Section \ref{section2}, we recall some background on
spin geometry and the Dirac operator, and prove elliptic estimates for the Dirac equation. In Section \ref{section3},
 we establish the three-circles estimates and prove Theorem \ref{thm:fixed}. In Section \ref{section4}, we recall the bubble tree
 convergence of non-collapsed degenerating Einstein manifolds and perform the bubble-neck decomposition.
In Section \ref{section5}, we prove Theorem \ref{thm:main}.

\vskip1cm

\section{Preliminaries}\label{section2}

\vskip0.5cm

\subsection{Spin structures and Dirac operator}
We briefly recall the fundamental notions from spin geometry, referring to the book \cite{LM}
for a comprehensive treatment. Let $(M,g)$ be an oriented Riemannian manifold of dimension
$n$. Denote by $P_{\SO}(M,g)$ the oriented frame bundle. A spin structure on $(M,g)$ is a
pair $(P_{\Spin}(M,g),\Lambda)$, where $P_{\Spin}(M,g)$ is a $\Spin(n)$-principal
bundle and $\Lambda:P_{\Spin}(M,g)\to P_{\SO}(M,g)$ is a 2-fold covering map compatible
with the non-trivial 2-fold covering $\lambda:\Spin(n)\to \SO(n)$.
Let
\[
        \tau:\Spin(n)\longrightarrow \operatorname{End}(\mathbb{S}_n )
\]
be the $n$-dimensional complex spinor representation, where $\mathbb{S}_n $ is a complex vector space
of dimension $2^{[n/2]}$. This induces an associated vector bundle
\[
        {\mathbb S(M)}=P_{\Spin}(M,g)\times_\tau \mathbb{S}_n ,
\]
which is called the spinor bundle over $M$. {In the sequel we use the notation $\mathbb S(M,g)$, or simply $\mathbb S(M)$ when the metric is clear, for the spinor bundle over $M$.} The spinor bundle ${\mathbb S(M)}$ carries a natural
Clifford multiplication $\cdot$, a Hermitian metric $\langle\cdot,\cdot\rangle$, and a spin
connection $\nabla^s$ such that:
\begin{enumerate}[label=(\roman*)]
\item For all $X,Y\in TM$ and all $\phi\in {\mathbb S(M)}$,
\[
        X\cdot Y\cdot\phi+Y\cdot X\cdot\phi+2g(X,Y)\phi=0.
\]
\item For all $X\in TM$ and all $\phi_1,\phi_2\in {\mathbb S(M)}$,
\[
        \langle X\cdot\phi_1,\phi_2\rangle=-\langle\phi_1,X\cdot\phi_2\rangle.
\]
\item $\nabla^s$ is compatible with the Hermitian metric $\langle\cdot,\cdot\rangle$ in the sense that
\[
        X\langle\phi_1,\phi_2\rangle
        =\langle\nabla^s_X\phi_1,\phi_2\rangle+\langle\phi_1,\nabla^s_X\phi_2\rangle.
\]
\item For all $X,Y\in\Gamma(TM)$ and all $\phi\in\Gamma({\mathbb S(M)})$,
\[
        \nabla^s_X(Y\cdot\phi)=(\nabla_XY)\cdot\phi+Y\cdot\nabla^s_X\phi.
\]
\end{enumerate}
On the spinor bundle ${\mathbb S(M)}$, the Dirac operator $D_M$ is defined as the composition
\[
        \Gamma({\mathbb S(M)})\xrightarrow{\nabla^s}\Gamma(T^*M\otimes {\mathbb S(M)})
        \longrightarrow \Gamma(TM\otimes {\mathbb S(M)})\xrightarrow{\cdot}\Gamma({\mathbb S(M)}).
\]
In particular, if $\{e_1,\ldots,e_{ n}\}$ is a local positively-oriented orthonormal basis of
$TM$, then the Dirac operator can be expressed as
\[
        D_M\psi=\sum_{j=1}^{ n} e_j\cdot\nabla^s_{e_j}\psi,
        \qquad \psi\in C^\infty(M,{\mathbb S(M)}).
\]

\subsubsection{Conformal changes of spinor bundles}

We fix the following convention for conformal changes, since it is used both in the cylinder transformation and in the bubble rescalings.  Let $\tilde g=e^{2u}g$ on a spin manifold or on the regular part of a spin orbifold.  The bundle map
\[
        b: (TM,g)\longrightarrow (TM,\tilde g),
        \qquad b(X)=e^{-u}X,
\]
induces an isomorphism of spinor bundles
\[
        \beta_{g,\tilde g}:\mathbb S(M,g)\longrightarrow\mathbb S(M,\tilde g),
\]
which is fiberwise unitary.  With our Clifford convention,
\[
        \beta_{g,\tilde g}(X\cdot_g\psi)
        =e^{-u}X\cdot_{\tilde g}\beta_{g,\tilde g}(\psi).
\]
The spin connections are related by
\[
\nabla^{s,\tilde g}_X\beta_{g,\tilde g}(\psi)-
\beta_{g,\tilde g}(\nabla^{s,g}_X\psi)
=-\frac12\beta_{g,\tilde g}\big(X\cdot_g\nabla^g u\cdot_g\psi+X(u)\psi\big),
\]
and the conformal covariance of the Dirac operator is
\begin{equation}\label{eq:conformal-dirac-general}
        D_{\tilde g}\Big(\beta_{g,\tilde g}(e^{-\frac{n-1}{2}u}\psi)\Big)
        =e^{-\frac{n+1}{2}u}\beta_{g,\tilde g}(D_g\psi).
\end{equation}
Consequently the critical spinor energy is conformally invariant:
\[
        \int_M |\psi|_g^{\frac{2n}{n-1}}\,dV_g
        =\int_M
        \left|\beta_{g,\tilde g}(e^{-\frac{n-1}{2}u}\psi)\right|_{\tilde g}^{\frac{2n}{n-1}}\,dV_{\tilde g}.
\]

\subsection{Dirac operator on product manifolds}
For the purposes of this paper, we next review the construction of spinor bundles and Dirac
operators on product manifolds. More detailed material can be found, for example, in
\cite{LM,Sire2023}. Let $(M_1,g_1)$ and $(M_2,g_2)$ be two spin manifolds of dimensions
$m_1$ and $m_2$, respectively. This induces a unique spin structure on the product manifold
$(N=M_1\times M_2,g=g_1\oplus g_2)$. Denote by ${\mathbb S(M_1)}$ and ${\mathbb S(M_2)}$ the spinor bundles
of $M_1$ and $M_2$. The spinor bundle ${\mathbb S(N)}$ on $N$ can be constructed as
\[
{\mathbb S(N)}=
\begin{cases}
  ({\mathbb S(M_1)}\oplus {\mathbb S(M_1)})\otimes {\mathbb S(M_2)}, & \text{both }m_1\text{ and }m_2\text{ are odd},\\
  {\mathbb S(M_1)}\otimes {\mathbb S(M_2)}, & m_1\text{ is even}.
\end{cases}
\]
For $X\in TM_1$, $Y\in TM_2$, $\varphi\in\Gamma({\mathbb S(M_2)})$, and
\[
\psi=
\begin{cases}
 \psi_1\oplus\psi_2\in\Gamma({\mathbb S(M_1)}\oplus {\mathbb S(M_1)}), & \text{both }m_1\text{ and }m_2\text{ are odd},\\
 \psi\in\Gamma({\mathbb S(M_1)}), & m_1\text{ is even},
\end{cases}
\]
the Clifford multiplication on ${\mathbb S(N)}$ is defined by
\begin{equation}\label{eq:prod-cliff}
        (X\oplus Y)\cdot_g(\psi\otimes\varphi)
        =(X\cdot_{g_1}\psi)\otimes\varphi
        +(\omega_C^{M_1}\cdot_{g_1}\psi)\otimes(Y\cdot_{g_2}\varphi),
\end{equation}
where $\omega_C^{M_1}$ is the complex volume element. When both $m_1$ and $m_2$ are odd,
we set
\[
        X\cdot_{g_1}\psi=(X\cdot_{g_1}\psi_1)\oplus(-X\cdot_{g_1}\psi_2),
        \qquad
        \omega_C^{M_1}\cdot_{g_1}\psi=i(\psi_2\oplus-\psi_1).
\]
It is worth noting that the definition of the Clifford multiplication on ${\mathbb S(N)}$ is not unique;
however, different choices are essentially equivalent.

Let $\nabla^{s_1}$ and $\nabla^{s_2}$ be the Levi-Civita spin connections on ${\mathbb S(M_1)}$ and
${\mathbb S(M_2)}$, respectively. Then
\[
        \nabla^s=\nabla^{s_1}\otimes\Id_{{\mathbb S(M_2)}}+
        \Id_{{\mathbb S(M_1)}}\otimes\nabla^{s_2}
\]
defines the tensor product connection. Choose local orthonormal frames
$\{X_1,\ldots,X_{m_1}\}$ for $(M_1,g_1)$ and $\{Y_1,\ldots,Y_{m_2}\}$ for $(M_2,g_2)$. Then
$\{X_1\oplus0,\ldots,X_{m_1}\oplus0,0\oplus Y_1,\ldots,0\oplus Y_{m_2}\}$ is a local
orthonormal frame on $N$. Using \eqref{eq:prod-cliff}, the Dirac operator on $N$ is given by
\[
\begin{aligned}
D_N&:=\sum_{j=1}^{m_1}(X_j\oplus0)\cdot_g\nabla^s_{X_j\oplus0}
      +\sum_{j=1}^{m_2}(0\oplus Y_j)\cdot_g\nabla^s_{0\oplus Y_j} \\
&=\widehat D_{M_1}\otimes\Id_{{\mathbb S(M_2)}}+
(\omega_C^{M_1}\cdot_{g_1}\Id_{{\mathbb S(M_1)}})\otimes D_{M_2},
\end{aligned}
\]
where
\[
\widehat D_{M_1}=
\begin{cases}
D_{M_1}\oplus(-D_{M_1}), & \text{both }m_1\text{ and }m_2\text{ are odd},\\
D_{M_1}, & m_1\text{ is even}.
\end{cases}
\]

\subsection{The Bourguignon-Gauduchon trivialization}
In this subsection, we employ the Bourguignon-Gauduchon trivialization and the conformal
invariance of the Dirac operator to transform the Dirac equation on a spin manifold $M$ into a
Dirac equation on an $n$-dimensional cylinder. The Bourguignon-Gauduchon trivialization
aims to locally identify spinors for different metrics on $M$. It was introduced by
Bourguignon and Gauduchon in \cite{BG}.

Fix a point $q\in M$ and let $(x_1,\ldots,x_n)$ be Riemannian normal coordinates defined
via the exponential map $\exp_q:U\subset T_qM\to V\subset M$. For
$p=\exp_q(x)\in V$, set $G(p)=(g_{ij}(p))$, where $g_{ij}(p)=g(\partial_i,\partial_j)$.
Define $B(p)=(b_i^j(p))$ as the unique symmetric positive-definite matrix satisfying
$B(p)^2=G(p)^{-1}$, which induces an isometry
\[
        B(p):(\R^n\simeq T_{\exp_q^{-1}(p)}U,g_{\R^n})\longrightarrow (T_pV,g(p))
\]
given by
\[
        B(p):(a^1,\ldots,a^n)\longmapsto \sum_{i,j}b_i^j(p)a^i\partial_j(p).
\]
This further induces a commutative diagram of  \(\mathrm{SO}(n)\)-principal bundles:
\[
\begin{tikzcd}
P_{\mathrm{SO}}(U, g_{\mathbb{R}^n}) \arrow[r, "\eta"] & P_{\mathrm{SO}}(V, g) \\
U \subset T_q M \arrow[r, "\exp_q"] \arrow[u, leftarrow] & V \subset M \arrow[u, leftarrow],
\end{tikzcd}
\]
where \(\eta(e_1, \dots, e_n) = (Be_1, \dots, Be_n)\)
for an oriented frame $(e_1, \dots, e_n)$ on $U$.
Since \(\eta\) is \(\mathrm{SO}(n)\)-equivariant, it can be lifted to an isomorphism between the corresponding  $\mathrm{Spin}(n)$-principal bundles:

\[
\begin{tikzcd}
    P_{\mathrm{Spin}}(U, g_{\mathbb{R}^n}) \arrow[r, "\bar{\eta}"] & P_{\mathrm{Spin}}(V, g) \\
P_{\mathrm{SO}}(U, g_{\mathbb{R}^n}) \arrow[r, "\eta"]\arrow[u, leftarrow] & P_{\mathrm{SO}}(V, g)\arrow[u, leftarrow] \\
U \subset T_q M \arrow[r, "\exp_q"] \arrow[u, leftarrow] & V \subset M \arrow[u, leftarrow].
\end{tikzcd}
\]
This isomorphism induces an isometric isomorphism between the spinor bundles:
\[
\begin{tikzcd}
\mathbb{S}(U, g_{\mathbb{R}^n}):=  P_{\mathrm{Spin}}(U, g_{\mathbb{R}^n})\times_\tau \mathbb{S}_n
 \arrow[r, "\bar{}"] & \mathbb{S}(V, g):= P_{\mathrm{Spin}}(V, g)\times_\tau \mathbb{S}_n
\end{tikzcd}
\]
explicitly given by
\[
\psi = [s, \varphi] \longmapsto  [\bar{\eta}(s), \varphi]=:\bar{\psi},
\]
where \([s, \varphi]\) denotes the equivalence class of \((s, \varphi)\) in \(P_{\mathrm{Spin}}(U, g_{\mathbb{R}^n})\times\mathbb{S}_n\).

Under the Bourguignon–Gauduchon trivialization, the transformation of the Dirac operator is
given by the following proposition, see Section 6.3 in \cite{Isobe} and Proposition 3.2 in \cite{AGHM}, for details.
\begin{proposition}\label{prop:BG}
Denote by $D_{\R^n}$ and $D_M$ the Dirac operators acting on
$\Gamma({\mathbb S(U,g_{\R^n})})$ and $\Gamma({\mathbb S(V,g)})$, respectively. Then
\begin{equation}\label{eq:BG-formula}
        D_M\bar\psi=
        \overline{D_{\R^n}\psi}+\mathbf{W}\cdot\bar\psi+\mathbf{V}\cdot\bar\psi+
        \sum_{i,j}\mathbf{Q}_{ij}\,\partial_i\cdot \overline{\nabla_{\partial_j}\psi},
\end{equation}
where
\[
        \mathbf{W}=\frac{1}{4}\sum_{\substack{i,j,k\\ i\ne j\ne k}}
        b_i^r(\partial_rb_j^l)(b^{-1})_l^k e_i\cdot e_j\cdot e_k=O(r^3),
\]
\[
        \mathbf{V}=\frac12\sum_{i,k}\widetilde\Gamma^i_{ik}e_k
        =-\frac14(\Ric)_{\alpha k}x^\alpha e_k+O(r^2)=O(r),
\]
and
\[
        \mathbf{Q}_{ij}=b_i^j-\delta_i^j=O(r^2).
\]
\end{proposition}

Using the Bourguignon-Gauduchon trivialization and the transformation formula above, we
convert equation \eqref{eq:intro-system} on the manifold into a Dirac equation on Euclidean
space. Assume that $\bar\psi\in\Gamma({\mathbb S(V)})^d$ satisfies \eqref{eq:intro-system}. Via the
trivialization, this corresponds to a spinor $\psi\in\Gamma({\mathbb S(U)})^d$ on the Euclidean domain
$U$. Applying Proposition~\ref{prop:BG}, we obtain an equivalent equation on $U$:
\begin{equation}\label{eq:Euclidean-approx}
        D_{\R^n}\psi^i=\sum_{j,l,t=1}^d H^i_{jlt}|\langle\psi^j,\psi^l\rangle|^{\frac1{n-1}}
        \psi^t+h^i,
\end{equation}
where
\begin{equation}\label{eq:h-fixed-bound}
        {|h|\le C r\big(|\psi|+|\nabla\psi|\big).}
\end{equation}

Next, we employ the conformal invariance of the Dirac operator to transform equation
\eqref{eq:Euclidean-approx} into a Dirac equation on a cylinder. For
$(t,\theta)\in[-\ln\delta,-\ln(\lambda_kR)]\times S^{n-1}$, let $(r,\theta)$ be polar coordinates
on $\R^n$ and define
\[
        f:\Sigma=[-\ln\delta,-\ln(\lambda_kR)]\times S^{n-1}\longrightarrow U,
        \qquad f(t,\theta)=(e^{-t},\theta),
\]
where $\Sigma$ is endowed with the product metric $g_{\cyl}=dt^2+d\theta^2$. The map $f$ is
a conformal diffeomorphism and $(f^{-1})^*g_{\cyl}=r^{-2}ds^2$.  This yields an isomorphism from \((T\Sigma, e^{-2t}g_{\mathrm{cyl}})\) to \((TU, g_{\mathbb{R}^n})\), which may be viewed as an isomorphism of \(\mathrm{SO}(n)\)-principal bundles. It lifts to an isomorphism of \(\mathrm{Spin}(n)\)-principal bundles and induces an isometric isomorphism

\[
F:\mathbb{S}(\Sigma, g_{\mathrm{cyl}})\longrightarrow \mathbb{S}(U, g_{\mathbb{R}^n}).
\]
Combining with the previous construction, we obtain the following commutative diagrams:

\[
\begin{tikzcd}
    P_{\mathrm{Spin}}(\Sigma, g_{\mathrm{cyl}})\arrow[r, "\bar{\zeta }"]&
    P_{\mathrm{Spin}}(U, g_{\mathbb{R}^n}) \arrow[r, "\bar{\eta}"] & P_{\mathrm{Spin}}(V, g) \\
    P_{\mathrm{SO}}(\Sigma, g_{\mathrm{cyl}}) \arrow[r, "\zeta"]\arrow[u, leftarrow] &
    P_{\mathrm{SO}}(U, g_{\mathbb{R}^n}) \arrow[r, "\eta"]\arrow[u, leftarrow] & P_{\mathrm{SO}}(V, g)\arrow[u, leftarrow] \\
\Sigma \arrow[r, "f"] \arrow[u, leftarrow] &
U \subset T_q M \arrow[r, "\exp_q"] \arrow[u, leftarrow] & V \subset M \arrow[u, leftarrow]
\end{tikzcd}
\]
and
\[
\begin{tikzcd}
    \mathbb{S}(\Sigma, g_{\mathrm{cyl}})\arrow[r, "F"]&
    \mathbb{S}(U, g_{\mathbb{R}^n})\arrow[r, "\bar{\eta}"]& \mathbb{S}(V, g) \\
\Sigma \arrow[r, "f"] \arrow[u, leftarrow] &
U \subset T_q M \arrow[r, "\exp_q"] \arrow[u, leftarrow] & V \subset M \arrow[u, leftarrow]
\end{tikzcd}
\]
Moreover, with respect to the metrics $g_{\R^n}$ and $g_{\cyl}$, the Dirac operators satisfy
\begin{equation}\label{eq:conformal-cylinder}
        D_\Sigma\Big(F^{-1}(e^{-\frac{n-1}{2}t}\psi)\Big)
        =e^{-\frac{n+1}{2}t}F^{-1}(D_{\R^n}\psi).
\end{equation}
Hence, if $\psi$ is a solution of \eqref{eq:Euclidean-approx}, then
$\phi=F^{-1}(e^{-\frac{n-1}{2}t}\psi)$ satisfies
\begin{equation}\label{eq:cylinder-equation}
        D_\Sigma\phi^i=\sum_{j,l,t=1}^d H^i_{jlt}|\langle\phi^j,\phi^l\rangle|^{\frac1{n-1}}
        \phi^t+\hat h^i,
\end{equation}
where $\hat h=e^{-\frac{n+1}{2}t}F^{-1}h$, and
\[
        \int_\Sigma |\phi|_{g_{\cyl}}^{\frac{2n}{n-1}}\,dV_{g_{\cyl}}
        =\int_U |\psi|_{g_{\R^n}}^{\frac{2n}{n-1}}\,dV_{g_{\R^n}}.
\]

\subsection{Elliptic estimates for the Dirac equation}
This subsection collects some results about elliptic estimates for the Dirac equation.
We begin with the following $\varepsilon$-regularity theorem for the Dirac equation.
\begin{lemma}\label{lem:epsilon-reg-fixed}
Let $n\ge 2$ and $p>1$. Let
\[
        \psi: B_{\delta}\setminus B_{r_jR}\longrightarrow \mathbb S(\mathbb R^n)
\]
be a solution of
\begin{equation}\label{eq:Dirac-annular}
        D_{\mathbb R^n}\psi
        =
        |\psi|^{\frac{2}{n-1}}\psi
        +
        a(y)\psi
        +
        b(y)\nabla\psi
        \qquad
        \text{on } B_{\delta}\setminus B_{r_jR},
\end{equation}
where $a(y)$ and $b(y)$ are smooth coefficient functions satisfying
\begin{equation}\label{eq:ab-assumption}
        |a(y)|+|b(y)|\le C_0 |y|
        \qquad
        \text{on } B_{\delta}.
\end{equation}
Then there exist constants $\delta_0$ and $\varepsilon_0$,
such that if $0<\delta\le \delta_0$ and
\(
        r_jR < r/2 < 5r/2 < \delta,
\)
and if
\begin{equation}\label{eq:small-annular-energy}
        \int_{B_{5r/2}\setminus B_{r/2}}
        |\psi|^{\frac{2n}{n-1}}\,dy
        <
        \varepsilon_0,
\end{equation}
then
\begin{equation}\label{eq:scale-invariant-W1p}
 r^{\frac{n-1}{2}-\frac np}
 \|\psi\|_{L^p(B_{2r}\setminus B_r)}
 +
 r^{\frac{n+1}{2}-\frac np}
 \|\nabla\psi\|_{L^p(B_{2r}\setminus B_r)}
 \le
 C_p
 \|\psi\|_{L^{\frac{2n}{n-1}}(B_{5r/2}\setminus B_{r/2})}.
\end{equation}
\end{lemma}

\begin{proof}
Set
\[
        q:=\frac{2n}{n-1}.
\]
We rescale the spinor by
\begin{equation}\label{eq:critical-scaling}
        \widetilde\psi(z)
        :=
        r^{\frac{n-1}{2}}\psi(rz),
        \qquad
        z\in B_{5/2}\setminus B_{1/2}.
\end{equation}
Since
\[
        \nabla_z\widetilde\psi(z)
        =
        r^{\frac{n+1}{2}}(\nabla_y\psi)(rz),
\]
the rescaled spinor satisfies
\begin{equation}\label{eq:scaled-Dirac}
        D_{\mathbb R^n}\widetilde\psi
        =
        |\widetilde\psi|^{\frac{2}{n-1}}\widetilde\psi
        +
        \widetilde a(z)\widetilde\psi
        +
        \widetilde b(z)\nabla\widetilde\psi
        \qquad
        \text{on } B_{5/2}\setminus B_{1/2},
\end{equation}
where
\begin{equation}\label{eq:scaled-coefficients}
        \widetilde a(z)=r\,a(rz),
        \qquad
        \widetilde b(z)=b(rz).
\end{equation}
By \eqref{eq:ab-assumption},
\begin{equation}\label{eq:small-scaled-coefficients}
        \|\widetilde a\|_{L^\infty(B_{5/2}\setminus B_{1/2})}
        +
        \|\widetilde b\|_{L^\infty(B_{5/2}\setminus B_{1/2})}
        \le
        C\delta.
\end{equation}
Furthermore, the $L^q$ norm is invariant:
\begin{equation}\label{eq:Lq-invariant}
        \|\widetilde\psi\|_{L^q(B_{5/2}\setminus B_{1/2})}
        =
        \|\psi\|_{L^q(B_{5r/2}\setminus B_{r/2})}.
\end{equation}
Hence \eqref{eq:small-annular-energy} implies
\begin{equation}\label{eq:small-rescaled-Lq}
        \|\widetilde\psi\|_{L^q(B_{5/2}\setminus B_{1/2})}
        \le
        \varepsilon_0^{1/q}.
\end{equation}
Let
\[
        B_{\rho_1^+}\setminus B_{\rho_1^-}
        \subset\!\subset
        B_{\rho_2^+}\setminus B_{\rho_2^-}
        \subset\!\subset
        B_{5/2}\setminus B_{1/2},
\]
and choose a cut-off function $\zeta\in C_c^\infty(B_{\rho_2^+}\setminus B_{\rho_2^-})$
with $\zeta\equiv1$ on $B_{\rho_1^+}\setminus B_{\rho_1^-}$.
For $1<s<n$, the local $L^s$ estimate for the Euclidean Dirac operator gives
\begin{align}\label{eq:local-dirac-estimate}
 \|\zeta\widetilde\psi\|_{W^{1,s}}
 &\le
 C_s\Big(
        \|\widetilde\psi\|_{L^s(B_{\rho_2^+}\setminus B_{\rho_2^-})}
        +
        \|\zeta |\widetilde\psi|^{\frac{2}{n-1}}\widetilde\psi\|_{L^s}
        \notag\\
 &\hspace{3.5cm}
        +
        \|\widetilde a\,\zeta\widetilde\psi\|_{L^s}
        +
        \|\widetilde b\,\zeta\nabla\widetilde\psi\|_{L^s}
        \Big).
\end{align}
Using Hölder and Sobolev inequalities,
the second term on the right-hand side of \eqref{eq:local-dirac-estimate} yields:
\begin{align}\label{eq:critical-absorption}
 \|\zeta |\widetilde\psi|^{\frac{2}{n-1}}\widetilde\psi\|_{L^s}
 &\le
 \||\widetilde\psi|^{\frac{2}{n-1}}\|_{L^n(B_{\frac52}\setminus B_{\frac12})}
 \|\zeta\widetilde\psi\|_{L^{\frac{ns}{n-s}}}
 \notag\\
 &\le
 C_s
 \|\widetilde\psi\|_{L^q(B_{\frac52}\setminus B_{\frac12})}^{\frac{2}{n-1}}
 \|\zeta\widetilde\psi\|_{W^{1,s}}.
\end{align}
 Consequently,
\begin{equation}\label{eq:fixed-annulus-W1s}
 \|\widetilde\psi\|_{W^{1,s}(B_{\rho_1^+}\setminus B_{\rho_1^-})}
 \le
 C_s
 \|\widetilde\psi\|_{L^s(B_{\rho_2^+}\setminus B_{\rho_2^-})},
 \qquad
 1<s<n.
\end{equation}
 Choose $s_0$ with
\[
        1<s_0<n,
        \qquad
        s_0\le q.
\]
When necessary, $s_0$ is chosen sufficiently close to $n$ so that the finite iteration below
reaches the exponents needed later.
By \eqref{eq:fixed-annulus-W1s} and the Sobolev embedding
$W^{1,s}\hookrightarrow L^{ns/(n-s)}$, we obtain an
improvement from $L^{s_0}$ to $L^{s_1}$, where
\[
        s_1=\frac{ns_0}{n-s_0}.
\]
Repeating this argument, for every $m>1$, we obtain
\begin{equation}\label{eq:high-integrability}
        \|\widetilde\psi\|_{L^m(B_{\frac94}\setminus B_{\frac34})}
        \le
        C_m
        \|\widetilde\psi\|_{L^q(B_{\frac52}\setminus B_{\frac12})}.
\end{equation}
The constants depend on $m,n,C_0$, but not on $r,j,R,\delta$.

Choose a cut-off function
$\chi\in C_c^\infty(B_{\frac94}\setminus B_{\frac34})$ with
$\chi\equiv1$ on $B_2\setminus B_1$. Applying the local $L^p$ estimate for the
Dirac operator to $\chi\widetilde\psi$ and using \eqref{eq:scaled-Dirac}, we get
\begin{align}\label{eq:final-Lp-estimate-start}
 \|\widetilde\psi\|_{W^{1,p}(B_2\setminus B_1)}
 &\le
 C_p\Big(
        \|\widetilde\psi\|_{L^p(B_{\frac94}\setminus B_{\frac34})}
        +
        \||\widetilde\psi|^{\frac{n+1}{n-1}}\|_{L^p(B_{\frac94}\setminus B_{\frac34})}
        \notag\\
 &\hspace{3.5cm}
        +
        \|\widetilde a\,\widetilde\psi\|_{L^p(B_{\frac94}\setminus B_{\frac34})}
        +
        \|\widetilde b\,\nabla\widetilde\psi\|_{L^p(B_{\frac94}\setminus B_{\frac34})}
        \Big).
\end{align}
As before, by taking
$\delta_0$ small if needed, we get
\begin{align}\label{eq:final-Lp-estimate}
 \|\widetilde\psi\|_{W^{1,p}(B_2\setminus B_1)}
 &\le
 C_p\Big(
        \|\widetilde\psi\|_{L^p(B_{\frac94}\setminus B_{\frac34})}
        +
        \|\widetilde\psi\|_{L^{\frac{p(n+1)}{n-1}}(B_{\frac94}\setminus B_{\frac34})}^{\frac{n+1}{n-1}}
        \Big).
\end{align}
Using \eqref{eq:high-integrability},
we obtain
\begin{equation}\label{eq:fixed-scale-final}
        \|\widetilde\psi\|_{W^{1,p}(B_2\setminus B_1)}
        \le
        C_p
        \|\widetilde\psi\|_{L^q(B_{\frac52}\setminus B_{\frac12})}.
\end{equation}
 This completes the proof.
\end{proof}
\begin{lemma}\label{lem:linear-eps}
Let $p>1$ and let $\psi:B_1\to {\mathbb S(\R^n)}$ be a solution of
\[
        D_{\R^n}\psi=|\phi|^{\frac2{n-1}}\psi+a(x)\psi+b(x)\nabla\psi,
\]
where $a(x)$ and $b(x)$ are smooth functions satisfying $a(x)=O(|x|)$ and
$b(x)=O(|x|)$, and $\|\phi\|_{L^{\frac{2n}{n-1}}}+\|\phi\|_{L^\infty}\le C$. Then
\[
        \|\psi\|_{W^{1,2}(B_{1/2})}\le C\|\psi\|_{L^2(B_1)},
\]
where the constant $C>0$ depends only on $n,p$, and the bound for
$\|\phi\|_{L^{\frac{2n}{n-1}}}+\|\phi\|_{L^\infty}$.
\end{lemma}

\begin{proof}
Let
\[
        V(x):=|\phi(x)|^{\frac2{n-1}} .
\]
By assumption, $\|V\|_{L^\infty(B_1)}$ is bounded.
Since $a(x)=O(|x|)$ and $b(x)=O(|x|)$, after shrinking the coordinate ball, or equivalently after applying the estimate on a sufficiently small rescaled ball, we may assume
\[
        \|b\|_{L^\infty(B_1)}\le \beta,
\]
where $\beta>0$ will be chosen below.  The coefficient $a$ satisfies
\[
        \|a\|_{L^\infty(B_1)}\le C.
\]
Choose $\eta\in C_c^\infty(B_{3/4})$ such that $\eta\equiv1$ on $B_{1/2}$, $0\le\eta\le1$, and $|\nabla\eta|\le C$.  Then
\[
        D_{\R^n}(\eta\psi)
        =\eta V\psi+\eta a\psi+\eta b\nabla\psi+\nabla\eta\cdot\psi .
\]
The ellipticity of the Dirac operator gives
\begin{align*}
\|\eta\psi\|_{W^{1,2}(\R^n)}
&\le C\Big(\|\eta\psi\|_{L^2}
        +\|\eta V\psi\|_{L^2}
        +\|\eta a\psi\|_{L^2}
        +\|\eta b\nabla\psi\|_{L^2}
        +\|\nabla\eta\cdot\psi\|_{L^2}\Big)  \\
&\le C\Big(1+\|V\|_{L^\infty(B_1)}+\|a\|_{L^\infty(B_1)}+\|b\|_{L^\infty(B_1)}\Big)
        \|\psi\|_{L^2(B_1)}
        +C\|b\|_{L^\infty(B_1)}\|\eta\nabla\psi\|_{L^2(B_1)} .
\end{align*}
Observe that
\[
        \eta\nabla\psi=\nabla(\eta\psi)-(\nabla\eta)\psi,
\]
so
\[
        \|\eta b\nabla\psi\|_{L^2}
        \le \|b\|_{L^\infty(B_1)}\|\nabla(\eta\psi)\|_{L^2}
        +C\|b\|_{L^\infty(B_1)}\|\psi\|_{L^2(B_1)} .
\]
Choosing $\beta>0$ so small that $C\beta\le 1/2$, the term involving
$\|\nabla(\eta\psi)\|_{L^2}$ is absorbed into the left-hand side.  Hence
\[
        \|\eta\psi\|_{W^{1,2}(\R^n)}
        \le C\|\psi\|_{L^2(B_1)}.
\]
Since $\eta\equiv1$ on $B_{1/2}$, this gives
\[
        \|\psi\|_{W^{1,2}(B_{1/2})}
        \le C\|\psi\|_{L^2(B_1)}.
\]
 This proves the lemma.
\end{proof}

\vskip1cm

\section{Energy quantization on fixed Spin manifolds}\label{section3}

\vskip0.5cm

 First, we prove a three circle theorem for Dirac operators. We then apply it to the
compactness analysis of the Dirac equation.

\subsection{Three circles theorem for Dirac operator}
At the beginning, we recall the eigenvalues
of the Dirac operator on the sphere given below. For the proof and details, see \cite{Bar}.

\begin{lemma}\label{lem:sphere-spectrum}
The Dirac operator $D_{S^{n-1}}$ on the sphere $S^{n-1}$ has the eigenvalues
\[
        \pm\left(\frac{n-1}{2}+k\right),\qquad k\ge0,
\]
with multiplicities
\[
        2^{\lfloor\frac{n-1}{2}\rfloor}\binom{k+n-2}{k}.
\]
For $\lambda\in {\Spec^+(D_{S^{n-1}}):=\{\frac{n-1}{2}+k:k\in\N\}}$, let
$\phi_{\lambda,1},\ldots,\phi_{\lambda,d(\lambda)}$ be eigen-spinors to eigenvalue $\lambda$.
Then the positive and negative eigenspinors form a complete orthonormal basis of
$L^2(S^{n-1},{\mathbb S(S^{n-1})})$.
\end{lemma}

\begin{lemma}\label{lem:harmonic-decomp}
Let $\phi$ be a harmonic spinor on $\Sigma$, then $\phi$ can be decomposed as
\[
        \phi=\sum_{\lambda\in\Spec(D_{S^{n-1}})}\sum_{j=1}^{d(\lambda)}
        \phi_{\lambda,j}\otimes f_{\lambda,j}
\]
provided $n-1$ is even. For $n-1$ is odd, $\phi$ can be decomposed as
\[
        \phi=\sum_{\lambda\in\Spec(D_{S^{n-1}})}\sum_{j=1}^{d(\lambda)}
        \phi_{\lambda,j}\otimes\binom{f_{\lambda,j}}{g_{\lambda,j}},
\]
where
\[
        f_{\lambda,j}(t)=A_{\lambda,j}e^{\lambda t}+B_{\lambda,j}e^{-\lambda t},
        \qquad
        g_{\lambda,j}(t)=C_{\lambda,j}e^{\lambda t}+D_{\lambda,j}e^{-\lambda t}.
\]
\end{lemma}

\begin{proof}
When $n-1$ is even,
 by  separation of variables, $\phi$ can be
  decomposed as \[\phi=\sum_{\substack{\lambda\in Spec(D_{S^{n-1}})\\ \lambda >0}}\sum_{j=1}^{d(\lambda)}\left[\phi_{\lambda,j}\otimes f_{\lambda,j}+\phi_{-\lambda,j}\otimes g_{\lambda,j}\right]. \]
 Then, using the fact that the chirality operator $\omega_C^{S^{n-1}}$ interchanges the $\lambda$ and $-\lambda$ eigenspaces of $D_{S^{n-1}}$,
 \[
 \omega_C^{S^{n-1}} \cdot \phi_{\pm\lambda,j}= \phi_{\mp\lambda,j},
 \]
 we obtain
  \begin{equation*}
      \begin{split}
         & 0= D_\Sigma\phi \\
        &  =
          \sum_{\substack{\lambda\in Spec(D_{S^{n-1}})\\ \lambda >0}}\sum_{j=1}^{d(\lambda)} \left[
  {(D_{S^{n-1}} \phi_{\lambda,j}) \otimes f_{\lambda,j}}
  + {(D_{S^{n-1}} \phi_{-\lambda,j}) \otimes g_{\lambda,j}}\right]
\\
&\quad +   \sum_{\substack{\lambda\in Spec(D_{S^{n-1}})\\ \lambda >0}}\sum_{j=1}^{d(\lambda)} \left[{(\omega_{\mathbb{C}}^{S^{n-1}} \cdot \phi_{\lambda,j}) \otimes \left(i \frac{df_{\lambda,j}}{dt}\right)}
  + {(\omega_{\mathbb{C}}^{S^{n-1}} \cdot \phi_{-\lambda,j}) \otimes \left(i \frac{dg_{\lambda,j}}{dt}\right)}
\right]\\
&=  \sum_{\substack{\lambda\in Spec(D_{S^{n-1}})\\ \lambda >0}}\sum_{j=1}^{d(\lambda)}
\left[
\phi_{\lambda,j}\otimes \left(\lambda f_{\lambda,j}+ i
\frac{dg_{\lambda,j}}{dt}\right)\right]\\
&\quad +
 \sum_{\substack{\lambda\in Spec(D_{S^{n-1}})\\ \lambda >0}}\sum_{j=1}^{d(\lambda)}
\left[
\phi_{-\lambda,j}\otimes \left(-\lambda g_{\lambda,j}+ i
\frac{df_{\lambda,j}}{dt}\right)\right],
      \end{split}
  \end{equation*}
 which yields $f_{\lambda,j}, g_{\lambda,j}: \mathbb{R} \to \mathbb{C}$ satisfy the coupled ODEs:
\[
\begin{cases}
   \lambda f_{\lambda,j}+ i
\frac{dg_{\lambda,j}}{dt}=0 \\[10pt]
   -\lambda g_{\lambda,j}+ i
\frac{df_{\lambda,j}}{dt}=0.
\end{cases}
\]
Then,
\[
f_{\lambda,j}(t) = A_{\lambda,j} e^{\lambda t} + B_{\lambda,j} e^{-\lambda t}, \quad
g_{\lambda,j}(t) =  \left( C_{\lambda,j} e^{\lambda t} + D_{\lambda,j} e^{-\lambda t} \right).
\]
Similarly, we obtain the decomposition for harmonic spinor provided $n-1$ is odd.
\end{proof}

Set $l_0L=-\ln\delta$, $l_iL=-\ln(\lambda_kR)$ and
$\Sigma_l=[(l-1)L,lL]\times S^{n-1}$ for $l=l_0,\ldots,l_i$.

\begin{lemma}\label{lem:three-circle-harmonic}
    Let $n\geq 3$ and $\phi$ be a harmonic spinor on $\Sigma$.
    Then there exists $L_0>0$, such that for $L>L_0$,
    there holds
    \[
{\|\phi\|_{L^2(\Sigma_l)}^2 < \frac{1}{2} \left( e^{-L} \|\phi\|_{L^2(\Sigma_{l+1})}^2 + e^{-L} \|\phi\|_{L^2(\Sigma_{l-1})}^2 \right)}
\]
\end{lemma}
\begin{proof}
We prove the case when $n-1$ is even (the odd case is analogous). By Lemma \ref{lem:harmonic-decomp}, $\phi$ decomposes as
\[
\phi = \sum_{\lambda \in \operatorname{Spec}(D_{S^{n-1}})} \sum_{j=1}^{d(\lambda)} \phi_{\lambda,j} \otimes f_{\lambda,j}, \quad f_{\lambda,j}(t) = A_{\lambda,j} e^{\lambda t} + B_{\lambda,j} e^{-\lambda t},
\]
where $|\lambda| \geq \frac{n-1}{2} > 0$ for all $\lambda \in \operatorname{Spec}(D_{S^{n-1}})$. By orthonormality of $\{\phi_{\lambda,j}\}$ in $L^2\left(S^{n-1}, \mathbb{S}(S^{n-1}) \right)$:
\[
\|\phi\|_{L^2(\Sigma_l)}^2 = \sum_{\lambda,j} \int_{(l-1)L}^{lL} |f_{\lambda,j}(t)|^2  dt = \sum_{\lambda,j} I_{\lambda,j}^{(l)},
\]
where
\begin{equation*}
\begin{split}
    I_{\lambda,j}^{(l)}  &= \int_{(l-1)L}^{lL} |A_{\lambda,j} e^{\lambda t} + B_{\lambda,j} e^{-\lambda t}|^2  dt\\
    &= \frac{|A_{\lambda,j}|^2 }{2\lambda} \left( e^{2\lambda l L} - e^{2\lambda (l-1)L} \right)
+ \frac{|B_{\lambda,j}|^2 }{2\lambda} \left( e^{-2\lambda (l-1)L} - e^{-2\lambda l L} \right)
+ 2\Re (A_{\lambda,j}B_{\lambda,j})L.
\end{split}
\end{equation*}
Set
\[
D := \frac{1}{2} \left( e^{-L} \|\phi\|_{L^2(\Sigma_{l+1})}^2 + e^{-L} \|\phi\|_{L^2(\Sigma_{l-1})}^2 \right) - \|\phi\|_{L^2(\Sigma_l)}^2 = \sum_{\lambda,j} K_{\lambda,j}^{(l)},
\]
where
\[
K_{\lambda,j}^{(l)}  = \frac{1}{2} e^{-L} \left( I_{\lambda,j}^{(l+1)}  + I_{\lambda,j}^{(l-1)}  \right) - I_{\lambda,j}^{(l)} .
\]
Denote $\alpha_\lambda = e^{2\lambda L} > 1$, we obtain:
\begin{align*}
K_{\lambda,j}^{(l)}
&= \frac{\alpha_\lambda - 1}{4\lambda} e^{-L} \left( |A_{\lambda,j}|^2  e^{2\lambda(l-1)L} Q_\lambda + |B_{\lambda,j}|^2  e^{-2\lambda l L} Q_\lambda \right) \\
&\quad + 2\Re(A_{\lambda,j}B_{\lambda,j})L(e^{-L} - 1),
\end{align*}
where $Q_\lambda = e^{2\lambda L} + e^{-2\lambda L} - 2e^L$.
 $|A_{\lambda,j}|^2  e^{2\lambda(l-1)L} + |B_{\lambda,j}|^2  e^{-2\lambda l L} \geq 2|A_{\lambda,j}B_{\lambda,j}|e^{-\lambda L}$ and $\alpha_\lambda-1\geq 2\lambda (e^L-1)$ gives
\[
K_{\lambda,j}^{(l)}  \geq |A_{\lambda,j}B_{\lambda,j}|(1-e^{-L}) \left[ Q_\lambda e^{-\lambda L}-2L
 \right]>0
\]
provided $L$ is large enough.

\end{proof}
\begin{remark}
   \(n\geq 3\) in Lemma \ref{lem:harmonic-decomp} is used in an essential way.
Indeed, in the proof we use that the positive eigenvalues of \(D_{S^{n-1}}\) satisfy
\[
\lambda\geq \frac{n-1}{2}\geq 1.
\]
This implies, for \(L\) sufficiently large,
\[
Q_\lambda:=e^{2\lambda L}+e^{-2\lambda L}-2e^L>0.
\]
When \(n=2\),  for
Neveu-Schwarz type spin structure on \(S^1\), the spectrum of Dirac operator is
\[
\operatorname{Spec}(D_{S^1})
=
\left\{\pm\left(k+\frac12\right): k=0,1,2,\ldots\right\}.
\]
In this case,
\[
Q_{1/2}=e^L+e^{-L}-2e^L=e^{-L}-e^L<0.
\]
Therefore the above proof does not give Lemma 3.3 in dimension two.

Nevertheless, in the Neveu--Schwarz case one still has three circles
inequality: for every \(0<\alpha<\frac12\), there exists \(L_0=L_0(\alpha)>0\)
such that, for all \(L>L_0\),
\[
\|\varphi\|^2_{L^2(\Sigma_l)}
\leq
\frac12 e^{-2\alpha L}
\left(
\|\varphi\|^2_{L^2(\Sigma_{l+1})}
+
\|\varphi\|^2_{L^2(\Sigma_{l-1})}
\right).
\]
Consequently, in dimension two , the conclusions
of Lemma \ref{lem:three-circle-perturbed} remain valid after replacing \(e^{-L/2}\) by \(e^{-\alpha L}\), where
\(0<\alpha<\frac12\). Correspondingly, in Lemma \ref{lem:three-circle-decay} the exponent
\(\hat\kappa=\min\{1,2\kappa\}\) should be replaced by
\[
\hat\kappa=\min\{2\alpha,2\kappa\}.
\]
\end{remark}

\subsection{Neck analysis on fixed Spin manifolds}
In this section, we prove the energy identity in Theorem~\ref{thm:fixed}
by establishing the following three circles theorem.
\begin{lemma}\label{lem:three-circle-perturbed}
Let $\phi:\Sigma_{l-1}\cup\Sigma_l\cup\Sigma_{l+1}\to {\mathbb S(\Sigma)}$ satisfy
\[
        D_\Sigma\phi=|\phi|^{\frac2{n-1}}\phi+\hat h,
\]
where $\hat h$ satisfies
\[
        |\hat h|\le C e^{-\kappa t}\big(|\phi|_{g_{\cyl}}+|\nabla_{\cyl}\phi|_{g_{\cyl}}\big)
\]
for some $\kappa>0$. Then there exist constants $\eps,\delta>0$ such that if
$$ \max\limits_{m=l-1,l,l+1}\|\phi\|_{L^{\frac{2n}{n-1}}(\Sigma_m)}<\eps$$ and
\begin{equation}\label{eq:h-small}
        \max_{m=l-1,l,l+1}\|\hat h\|_{L^p(\Sigma_m)}\le \delta\|\phi\|_{L^2(\Sigma_l)},
\end{equation}
the following hold:
\begin{enumerate}[label=(\alph*)]
\item $\|\phi\|_{L^2(\Sigma_{l+1})}\le e^{-L/2}\|\phi\|_{L^2(\Sigma_l)}$ implies
$\|\phi\|_{L^2(\Sigma_l)}\le e^{-L/2}\|\phi\|_{L^2(\Sigma_{l-1})}$;
\item $\|\phi\|_{L^2(\Sigma_{l-1})}\le e^{-L/2}\|\phi\|_{L^2(\Sigma_l)}$ implies
$\|\phi\|_{L^2(\Sigma_l)}\le e^{-L/2}\|\phi\|_{L^2(\Sigma_{l+1})}$;
\item either $\|\phi\|_{L^2(\Sigma_l)}\le e^{-L/2}\|\phi\|_{L^2(\Sigma_{l+1})}$ or
$\|\phi\|_{L^2(\Sigma_l)}\le e^{-L/2}\|\phi\|_{L^2(\Sigma_{l-1})}$.
\end{enumerate}
\end{lemma}

\begin{proof}
We prove (a) by contradiction; the proofs of (b) and (c) are the same. Suppose (a) fails.
Then there exist sequences $\eps_k,\delta_k\to0$ and spinors $\phi_k$ satisfying
\[
        D_\Sigma\phi_k=|\phi_k|^{\frac2{n-1}}\phi_k+h_k,
        \qquad
         \max_{m=l-1,l,l+1}\|\phi_k\|_{L^{\frac{2n}{n-1}}(\Sigma_m)}\to0,
\]
with $\max\limits_{l-1,l,l+1}\|h_k\|_{L^p}\le\delta_k\|\phi_k\|_{L^2(\Sigma_l)}$, but
\[
        \|\phi_k\|_{L^2(\Sigma_l)}>e^{-L/2}\|\phi_k\|_{L^2(\Sigma_{l+1})},
        \qquad
        \|\phi_k\|_{L^2(\Sigma_l)}>e^{-L/2}\|\phi_k\|_{L^2(\Sigma_{l-1})}.
\]
Set $\widetilde{\phi} _k=\phi_k/\|\phi_k\|_{L^2(\Sigma_l)}$. By Lemmas~\ref{lem:epsilon-reg-fixed} and
\ref{lem:linear-eps}, $\widetilde{\phi} _k$ is uniformly bounded in $W^{1,2}_{\mathrm{loc}}$ and, after passing to a
subsequence, converges strongly in $L^2(\Sigma_l)$ and weakly in $L^2$ on the three adjacent
cylinders to a harmonic spinor $\widetilde{\phi} $. The above inequalities pass to the limit and contradict
Lemma~\ref{lem:three-circle-harmonic}.
\end{proof}

The following lemma is a direct corollary of Lemma~\ref{lem:three-circle-perturbed}.

\begin{lemma}\label{lem:three-circle-decay}
Let $\phi:\Sigma\to \mathbb{S}(\Sigma)$ satisfy
\[
        D_\Sigma\phi=|\phi|^{\frac2{n-1}}\phi+\hat h,
\]
where $\hat h$ satisfies
\[
        |\hat h|\le C e^{-\kappa t}(|\phi|_{g_{\cyl}}+|\nabla_{\cyl}\phi|_{g_{\cyl}})
\]
for some $\kappa>0$, and assume that for every $l=l_0,\ldots,l_i$,
$\|\phi\|_{L^{\frac{2n}{n-1}}(\Sigma_l)}<\eps$. Then the following estimate holds:
\[
        \|\phi\|_{L^2(\Sigma_l)}^2
        \le C\eps^2\left(e^{-\hat\kappa(l-l_0)L}+e^{-\hat\kappa(l_i-l)L}\right)
\]
for $\hat\kappa = \min\{ 1, 2\kappa\}$.
\end{lemma}

\begin{proof}[Proof of Theorem~\ref{thm:fixed}]
By the standard induction argument in \cite{DingTian}, we only need to prove the theorem in the case
where there is only one bubble. For two
or more bubbles forming a bubble tree, the proof follows by repeating the same neck argument
between adjacent scales. The proof of Theorem~\ref{thm:fixed} is equivalent to proving
\begin{equation}\label{eq:fixed-neck-zero}
        \lim_{\delta\to0}\lim_{R\to\infty}\lim_{k\to\infty}
        \int_{B_\delta\setminus B_{\lambda_kR}} |\psi_k|^{\frac{2n}{n-1}}\,dV_g=0.
\end{equation}
Following the Ding and Tian's paper \cite{DingTian}, under the one bubble assumption,
for any $\eps>0$, when $k$, $R$ and $1/\delta$ are large enough, each cylinder
$\Sigma_l$ in the neck satisfies
\[
        \|\phi_k\|_{L^{\frac{2n}{n-1}}(\Sigma_l)}<\eps,
\]
where $\phi_k:=F^{-1}(e^{-\frac{n-1}{2}t}\psi_k)$.
By Lemma~\ref{lem:epsilon-reg-fixed}, $\|\psi_k\|_{L^\infty(\Sigma_l)}\le C\eps$, and the same estimate holds for $\phi_k$. Then, using Lemma~\ref{lem:three-circle-decay},
\[
\begin{aligned}
\int_{B_\delta\setminus B_{\lambda_kR}} |\psi_k|^{\frac{2n}{n-1}}\,dV_g
&\leq C\int_\Sigma |\phi_k|^{\frac{2n}{n-1}}\,dV_{g_{\cyl}}\\
&\le C\eps^{\frac2{n-1}}\int_\Sigma |\phi_k|^2\,dV_{g_{\cyl}}\\
&\le C\eps^{\frac2{n-1}}
\sum_{l=l_0}^{l_i}\eps^2\left(e^{-\hat\kappa(l-l_0)L}+e^{-\hat\kappa(l_i-l)L}\right).
\end{aligned}
\]
Letting $k\to\infty$, $R\to\infty$, $\delta\to0$, and finally $\eps\to0$, we obtain
\eqref{eq:fixed-neck-zero}.
\end{proof}

\vskip1cm

\section{Bubble-neck decomposition in the degenerating case}\label{section4}

\vskip0.5cm

We first recall the compactness theory for non-collapsed Einstein $n$-manifolds. Consider a
sequence $(M_k,g_k)$ of closed $n$-dimensional Einstein manifolds with uniformly bounded
Einstein constants. Assume the manifolds are non-collapsed:
\[
        \diam(M_k,g_k)\le D,
        \qquad \vol(M_k,g_k)\ge V>0,
\]
and their curvature tensors $\Rm_{g_k}$ have uniformly bounded $L^{n/2}$ norms, i.e.
\[
        \int_{M_k}|\Rm_{g_k}|^{\frac n2}\,dV_{g_k}\le R.
\]
By classical works, we can suppose that $(M_k,g_k)$ converges smoothly except at the set
$S$ of singular points up to a subsequence. Then for every $x_a\in S$, there is a point
$x_{a,k}\in M_k$ such that for a positive constant $r_\infty$ sufficiently small,
\[
        \sup_{B(x_{a,k},r_\infty)}|\Rm_{g_k}|=|\Rm_{g_k}|(x_{a,k})\to\infty
        \quad\text{as }k\to\infty,
\]
and
\[
        \int_{B(x_a,r_\infty)}|\Rm_{g_\infty}|^{\frac n2}\le \frac\eps2
\]
with a small positive number $\eps\le\bar\eps/2$. For sufficiently large $k$ we can find a
positive number $r_k>0$ so that
\[
        \int_{B(x_{a,k},r_\infty)\setminus B(x_{a,k},r_k)}
        |\Rm_{g_k}|^{\frac n2}\,dV_{g_k}=\eps.
\]
It is easy to see that $r_k\to0$ as $k\to\infty$. In the above $\bar{\varepsilon}$ is the positive
 number determined by the small energy regularity theorem for Einstein metrics
  (see e.g. \cite{Anderson1989,Tian,Anderson1992,Bando,CheegerTian})

\begin{theorem}[\cite{Nakajima1988,Anderson1989, BKN, Tian,  Bando, Anderson1992,Anderson1995, Nakajima1994}]\label{thm:einstein-compactness}
Let $(M_k,g_k)$ be a sequence of closed Einstein {$n$-manifolds} as above. Then there
exists a subsequence, still denoted by $k$, and a compact Einstein orbifold $(M_\infty,g_\infty)$
with a finite set of orbifold singular points $S=\{x_1,\ldots,x_s\}\subset M_\infty$ for which
the following statements hold:
\begin{enumerate}[label=(\arabic*)]
\item $(M_k,g_k)$ converges to $(M_\infty,g_\infty)$ in the Gromov-Hausdorff distance.
There exists a diffeomorphism $F_k:M_\infty\setminus S\to M_k$ onto its image such that
$F_k^*g_k$ converges to $g_\infty$ in the $C^\infty$ topology on compact subsets of
$M_\infty\setminus S$.
\item For every $x_a\in S$ and $k$, let $x_{a,k}\in M_k$ and $r_k$ be chosen as above. Then:
\begin{enumerate}[label=(2.\alph*)]
\item $B(x_{a,k},\delta)$ converges to $B(x_a,\delta)$ in the Gromov-Hausdorff distance
for all $\delta>0$.
\item Up to a subsequence, $((M_k,r_k^{-2}g_k),x_{a,k})$ converges to
$((Y,h),y_\infty)$ in the pointed Gromov-Hausdorff distance, where $(Y,h)$ is a complete,
Ricci-flat, non-flat ALE {$n$-orbifold} with a finite singular set. The convergence is smooth
outside the singular points.
\item There exist positive constants $C>0$ and $\eps_5>0$ such that for
$4r_k\le r<4r\le r_\infty$,
\[
        r^2|\Rm_{g_k}|
        \le C\max\left\{\left(\frac{r_k}{r}\right)^{\eps_5},
        \left(\frac{r}{r_\infty}\right)^{\eps_5}\right\}.
\]
\end{enumerate}
\item In the case that there is only one ALE bubble manifold at each singular point, we have:
\begin{enumerate}[label=(3.\alph*)]
\item $((M_k,r_k^{-2}g_k),x_{a,k})$ converges to $((M_a,h_a),x_{a,\infty})$, where
$(M_a,h_a)$ is a Ricci-flat ALE {$n$-manifold or orbifold}.
\item There exists a diffeomorphism $G_k:M_a\to M_k$ onto its image such that
$G_k^*(r_k^{-2}g_k)$ converges to $h_a$ in the $C^\infty$ topology on compact subsets.
\end{enumerate}
\item If there are several bubble manifolds or orbifolds at some singular points, by repeating
the same process, one obtains a bubble tree $\operatorname{Tr}_{\ALE}$ consisting of finitely
many Ricci-flat ALE bubble spaces.
\item Curvature energy identity:
\[
        \lim_{k\to\infty}\int_{M_k}|\Rm_{g_k}|^{\frac n2}\,dV_{g_k}
        =\int_{M_\infty}|\Rm_{g_\infty}|^{\frac n2}\,dV_{g_\infty}
        +\sum_j\int_{X_j}|\Rm_{h_j}|^{\frac n2}\,dV_{h_j},
\]
where $\{(X_j,h_j)\}$ is the set of all ALE bubble spaces in the tree.
\end{enumerate}
Consequently, for $1<K_1<K_2$ sufficiently large, the annulus
\[
        \big(B(x_{a,k},K_2^{-1}r_\infty)\setminus B(x_{a,k},K_1r_k),r_k^{-2}g_k\big)
\]
is close to a portion of the flat cone $\R^n/\Gamma$ for large $k$, where
$\Gamma$ is a finite group in $\SO(n)$.
\end{theorem}

There is a refined geometric picture of the degenerating neck
 region $A_{r_kR,\delta}(x_{a,k})$. In \cite{ChenZhu2024}, the authors construct two types of good global coordinates on this region: $(x)$ and $(y)$ in the case $n=4$. In the coordinates $(x)$, the metrics $g_j$ are $C^0$ close to the flat metric. While, in the coordinates $(y)$, the metrics $g_j$ are $C^m$ close to the flat metric in some weighted function space for any $m\geq  0$.

\begin{theorem}\label{thm:neck-coordinates}
For  $\delta_1>0$ small enough and for $R>0$ and $k>0$ sufficiently large, there exist two coordinates $(x)$ and $(y)$ on $A_{r_kR,\delta_1}(x_{a,k})\subset M_k$ such that the following properties hold:
in coordinates $(x)$, there hold
\begin{eqnarray*}
& &|g_{k,jl}( x)-\delta_{jl}|<C \eta_k(|x|),\\
& &|x(\cdot)|\equiv r(\cdot)=d_{g_k}(\cdot,x_{a,k})\quad \text{and} \quad \nabla_{\partial_r}\partial_r=0,
\end{eqnarray*}
and in coordinates $(y)$, there holds for some $0<\alpha<1$ that
\begin{equation*}
\|g_{k,jl}( y)-\delta_{jl}\|_{C^{4,\alpha}_{\eta_k(|y|)}}<C,
\end{equation*}
where $\eta_k(r)= \left(\left(\frac{r_k}{r}\right)^{\varepsilon_5}+\left(\frac{r}{r_\infty}\right)^{\varepsilon_5}  \right),$ $r_\infty>0$, $\varepsilon_5 >0$, $r_k>0$ are the same as in the curvature estimates in Theorem \ref{thm:einstein-compactness}.
Moreover,
\begin{equation*}
\|y-x\|_{C^{1}_{|y|\eta_k(|y|)}}<C.
\end{equation*}
\end{theorem}

We remark that the arguments in \cite{ChenZhu2024} can be modified without any essential difficulty to give the proof of the above theorem in the $n>4$ case.

Here we say coordinates on the degenerating neck region,
we mean a map
\[\varphi=\mathcal{L}^{-1}\circ proj:\mathbb{R}^n \rightarrow  A_{r_kR,\delta_1}(x_{a,k}),\]
where $\mathcal{L}$ is a diffeomorphism from $ A_{r_kR,\delta_1}(x_{a,k})$
to $\mathbb{R}^n/\Gamma$, and $proj$ is the natural projection from
$\mathbb{R}^n$ to $\mathbb{R}^n/\Gamma$.  The weighted H\"{o}lder space
 $C^{m,\alpha}_{\eta}(A_{r_1,r_2})$ ($0<\alpha<1$) with weight $\eta$
 (a function of $|x|=r$) is the space of functions in
 $C^{m,\alpha}(A_{r_1,r_2}(0))$ ($0\in\mathbb{R}^n$) with bounded weighted
 norm $||\cdot||_{C^{k,\alpha}_{\eta}}$, which is defined as follows:
\begin{eqnarray*}
||f||_{C^{m,\alpha}_{\eta}}&=&\sum_{j=0}^{m}\sup_{A_{r_1,r_2}(0)}\eta^{-1}|x|^{j}|D^{j}f|\\
& &+\sup_{x\neq y}
\min \ (\eta^{-1}|x|^{m+\alpha},\eta^{-1}|y|^{m+\alpha})\frac{|D^{m}f(x)-D^{m}f(y)|}{|x-y|^{\alpha}}.
\end{eqnarray*}

\subsection{Blow-up away from orbifold singularity}\label{sec:blowup-away-orbifold}
We first consider the case where the blow-up occurs away from the orbifold singularities. By
Theorem~\ref{thm:einstein-compactness}, outside the singular point $x_a$, the sequence
$(M_k,g_k)$ converges smoothly to $(M_\infty,g_\infty)$. For any neighborhood $U$ of $x_a$,
since $H^1(M_\infty\setminus U;\Z_2)$ is finite, after passing to a subsequence we may assume
that each diffeomorphism $F_k$ is spin compatible. Via $F_k$, we regard $\psi_k$ as a spinor
field $\phi_k$ on $M_\infty\setminus U$. Because $F_k^*g_k$ converges to $g_\infty$ in the
$C^\infty$ topology on $M_\infty\setminus\{x_a\}$, a diagonal argument shows that $\phi_k$
converges weakly to a spinor $\psi_\infty$ on $M_\infty\setminus\{x_a\}$, and $\psi_\infty$
satisfies the limiting Dirac system.

The spin structure on $M_\infty\setminus\{x_a\}$ extends to an orbifold spin structure on
$M_\infty$ whenever the local isotropy lifts to $\Spin(n)$.
More precisely, in an orbifold coordinate neighborhood $U\simeq B^n/\Gamma$ of $x_a$, with
$\Gamma\subset\SO(n)$ finite, we choose a lift
$\widehat\Gamma\subset\Spin(n)$ such that
$\lambda|_{\widehat\Gamma}:\widehat\Gamma\to\Gamma$ is an isomorphism.  Locally
\[
        \mathbb S(M_\infty)|_U\simeq (B^n\times \mathbb{S}_n )/\widehat\Gamma,
\]
and the orbifold Dirac operator is the one induced from the
$\widehat\Gamma$-equivariant Dirac operator on the covering ball.  This is the spin structure
obtained from the smooth spin structures on $M_k$ through the spin-compatible
Cheeger-Gromov embeddings and the extension theorem for orbifolds whose singular set has
codimension at least four \cite{BelgunGinouxRademacher,KapovitchLott}.
Let $\pi:B^n\setminus\{0\}\to U\setminus\{x_a\}$ be the covering map. Then the pull-back
spinor $\widetilde\psi_\infty=\pi^*\psi_\infty$ satisfies the same equation on
$B^n\setminus\{0\}$ with respect to the pulled-back metric $\pi^*g_\infty$. By the
removable singularity theorem, $\widetilde\psi_\infty$ extends smoothly to the
whole ball $B^n$. Consequently, $\psi_\infty$ is continuous near $x_a$ as an orbifold spinor.

Next, suppose that $\psi_k$ blows up at a nondegenerate point $x_{b,k}$, i.e. the convergence
$(B(x_{b,k},r),g_k)\to(B(x_b,r),g_\infty)$ is smooth in the Cheeger-Gromov sense. Then,
via the smooth diffeomorphisms, we may locally view $\psi_k$ as a spinor field defined on
$B(x_b,r)$. Without loss of generality we assume that, away from the orbifold singularity
$x_a$, the sequence $\psi_k$ blows up only at the point $x_b$ and produces exactly one bubble
 $\xi^b$. In local coordinates near the blow-up point $x_b$, there exists a sequence
$\lambda_k\to0$ such that
\[
        {\lambda_k^{\frac{n-1}{2}}\psi_k(\lambda_k x)\longrightarrow \xi^b,}
\]
where
\[
        {\xi^b: \R^n\to \Gamma({\mathbb S(\R^n)})^d}
\]
is a solution of Dirac system \eqref{eq:intro-system} with respect to the Euclidean metric.
Following an argument analogous to the proof of Theorem~\ref{thm:fixed}, we obtain the
following bubble tree convergence result away from the orbifold singularities.

\begin{theorem}\label{thm:away-orbifold}
Let $\psi_k$ be a sequence of spinor fields on $(B(x_b,1),g_k)$ satisfying the critical Dirac
system and assume that for some constant $C>0$,
\[
        E(\psi_k,B_1)=\int_{B_1}|\psi_k|^{\frac{2n}{n-1}}_{g_k}\,dV_{g_k}\le C.
\]
If $\{\psi_k\}$ blows up at $x_b$ away from the orbifold singularity and produces exactly one
non-trivial bubble {$\xi^b\in\Gamma({\mathbb S(\R^n)})^d$}, then
\[
        \lim_{\delta\to0}\lim_{R\to\infty}\lim_{k\to\infty}
        \int_{B_\delta\setminus B_{\lambda_kR}} |\psi_k|^{\frac{2n}{n-1}}\,dV_{g_k}=0.
\]
\end{theorem}

\subsection{Blow-up at the orbifold singularity}\label{sec:blowup-orbifold}

We now investigate the blow-up of the sequence $\psi_k$ at an orbifold singularity
$x_a\in M_\infty$. For simplicity, we assume that there is only one ALE bubble space $M_a$
at $x_a$. By Theorem~\ref{thm:einstein-compactness}, for each $k$ there exist
$x_{a,k}\in M_k$ and a sequence of positive numbers $r_k\to0$ such that
\[
        ((M_k,r_k^{-2}g_k),x_{a,k})\longrightarrow ((M_a,h_a),x_{a,\infty})
\]
in the Gromov-Hausdorff distance, where $(M_a,h_a)$ is a complete, noncompact,
Ricci-flat and non-flat {$n$-dimensional} ALE space. Moreover, there exist diffeomorphisms
$G_k:M_a\to M_k$ such that $G_k^*(r_k^{-2}g_k)$ converges to $h_a$ in the $C^\infty$
topology on compact subsets of $M_a$.

The pull-back spin structures $G_k^*\sigma_k$ on compact subsets of $M_a^{\reg}$ determine,
after passing to a subsequence, a spin structure on the regular part of the ALE limit.  As in
Section~\ref{sec:blowup-away-orbifold}, this structure extends across the orbifold points of
$M_a$.  Thus $M_a$ is regarded as a Ricci-flat ALE spin orbifold, and $D_{h_a}$ is defined
by lifting to local uniformizing charts and descending the equivariant Dirac operator.  On the
ALE end $M_a\setminus K\simeq (\R^n\setminus B_R)/\Gamma_a$, the inherited lift
$\widehat\Gamma_a\subset\Spin(n)$ also defines the spinor bundle and the descended Euclidean
Dirac operator on the cone $\R^n/\Gamma_a$.

For any $R>0$, $B(x_{a,k},r_kR)$ corresponds
to the ALE bubble domain, $M_k\setminus B(x_{a,k},\delta)$ is the base domain, and
\[
        A_{r_kR,\delta}(x_{a,k})=B(x_{a,k},\delta)\setminus B(x_{a,k},r_kR)
\]
is the degenerating neck region. At the orbifold singularity $x_a$, several bubbles may occur.
Assume that
\[
        {(\lambda_k^j)^{\frac{n-1}{2}}\psi_k:
        (M_k,(\lambda_k^j)^{-2}g_k,x_{a,j,k})\longrightarrow \xi^j}
\]
converges, as $k\to\infty$, to a bubble spinor $\xi^j$, and that
\[
        \lim_{k\to\infty}d_{g_k}(x_{a,j,k},x_{a,k})=0.
\]
According to the relative size of the blow-up scales, we distinguish the following three cases:
\begin{description}[leftmargin=0pt,labelsep=0.5em,itemsep=0.4em]
\item[\textbf{Case A}]
\[
\lim_{k\to\infty} \frac{d_{g_k}(x_{a,j,k}, x_{a,k})}{r_k} < \infty,
\qquad
\lim_{k\to\infty} \frac{\lambda_k^j}{r_k} = 0.
\]

\item[\textbf{Case B}]
\[
\lim_{k\to\infty} \frac{d_{g_k}(x_{a,j,k}, x_{a,k})}{r_k} < \infty,
\qquad
0 < \lim_{k\to\infty} \frac{\lambda_k^j}{r_k} < \infty.
\]

\item[\textbf{Case C}]
others.
\end{description}

In what follows, the bubbles in Cases A, B and C are denoted by
{$\xi^A$, $\xi^B$ and $\xi^C$}, respectively.

In Case A, we have
\[
        {(\lambda_k^A)^{\frac{n-1}{2}}\psi_k:
        (M_k,(\lambda_k^A)^{-2}g_k,x^A_{a,k})\longrightarrow
        \xi^A\in\Gamma({\mathbb S(\R^n)})^d.}
\]
By an argument analogous to that of Theorem~\ref{thm:fixed}, there is no energy loss in the
neck region
\[
        A_{\lambda_k^AR,\delta r_k}(x^A_{a,k})
        =B(x^A_{a,k},\delta r_k)\setminus B(x^A_{a,k},\lambda_k^AR).
\]
In Case B, without loss of generality, we may assume that $x_{a,j,k}=x_{a,k}$ and
$r_k=\lambda_k^B$. Then
\[
        {(\lambda_k^B)^{\frac{n-1}{2}}\psi_k:
        (M_k,(\lambda_k^B)^{-2}g_k,x_{a,k})\longrightarrow
        \xi^B\in\Gamma({\mathbb S(M_a,h_a)})^d.}
\]
For Case C, a natural question arises: what is the behavior of $\psi_k$ on the neck region
\[
        A_{r_kR,\delta}(x_{a,k})=B(x_{a,k},\delta)\setminus B(x_{a,k},r_kR)
\]
as $k\to\infty$?

\begin{remark}\label{rem:A-on-B}
By our construction, bubble {$\xi^A$} lies on the bubble {$\xi^B$}, and by applying
the same arguments as in Theorem~\ref{thm:fixed}, we know that there is no energy loss in
the neck region
\[
        A_{\lambda_k^AR,\delta r_k}(x^A_{a,j})
        =B(x^A_{a,k},\delta r_k)\setminus B(x^A_{a,k},\lambda_k^AR).
\]
\end{remark}

Motivated by the blow-up analysis of biharmonic maps on non-collapsed degenerating Einstein
manifolds and of harmonic maps on degenerating Riemann surfaces, we set
$T_k=\log(r_kR)$, $T(\delta)=\log\delta$, and obtain the following bubble-neck decomposition.

\begin{proposition}\label{prop:bubble-neck}
With notations and assumptions above.
\begin{enumerate}[label=(\arabic*)]
\item Asymptotic boundary conditions:
\[
        \lim_{k\to\infty}\omega(\psi_k,P_{T_k,T_k+L})=0,
        \qquad
        \lim_{\delta\to0}\omega(\psi_k,P_{T(\delta)-L,T(\delta)})=0,
        \qquad \forall L\ge1,
\]
where $P_{T_1,T_2}$ is the cylinder corresponding to
$A_{e^{T_1},e^{T_2}}(x_{a,k})\subset M_k$, and
\[
        \omega(\psi_k,P_{T_1,T_2})
        =\sup_{t\in[T_1,T_2-1]}
        \int_{A_{e^t,e^{t+1}}(x_{a,k})}|\psi_k|^{\frac{2n}{n-1}}_{g_k}\,dV_{g_k}.
\]
\item Bubble domain and neck domain: after selection of a subsequence, which we still denote
by $\psi_k$, the following two alternatives hold:
\begin{enumerate}[label=(2.\arabic*)]
\item
\[
        \lim_{\delta\to0}\lim_{R\to\infty}\lim_{k\to\infty}\omega(\psi_k,P_{T_k,T(\delta)})=0.
\]
\item There exists some number $N_2>0$ independent of $k$, and $2N_2$ sequences of numbers
$\{a_k^1\},\{b_k^1\},\ldots,\{a_k^{N_2}\},\{b_k^{N_2}\}$ such that
\[
        T_k\le a_k^1\ll b_k^1\le\cdots\le a_k^{N_2}\ll b_k^{N_2}\le T(\delta),
        \qquad a_k^\alpha\ll b_k^\alpha\text{ means }b_k^\alpha-a_k^\alpha\to\infty,
\]
and
\[
        |b_k^\alpha-a_k^\alpha|\ll |T_k|,
        \qquad
        \lim_{k\to\infty}\frac{|b_k^\alpha-a_k^\alpha|}{|T_k|}=0.
\]
Denote
\[
        J_k^\alpha=P_{a_k^\alpha,b_k^\alpha},\qquad \alpha=1,\ldots,N_2,
\]
\[
        I_k^0=P_{T_k,a_k^1},\qquad
        I_k^{N_2}=P_{b_k^{N_2},T(\delta)},\qquad
        I_k^\alpha=P_{b_k^\alpha,a_k^{\alpha+1}},
        \quad \alpha=1,\ldots,N_2-1.
\]
Then:
\begin{enumerate}[label=(\roman*)]
\item[({\romannumeral1})] for every $\alpha=0,1,\ldots,N_2$,
\[
        \lim_{\delta\to0}\lim_{R\to\infty}\lim_{k\to\infty}\omega(\psi_k,I_k^\alpha)=0;
\]
\item[({\romannumeral2})] for every $\alpha=1,\ldots,N_2$, there is a bubble tree which consists of at most
finitely many finite-energy bubble spinors on $\R^n/\Gamma$. For simplicity, assume there is
only one such bubble spinor, namely,
\[
        {\xi^{C,\alpha}\in\Gamma({\mathbb S(\R^n/\Gamma)})^d,}
\]
such that
\[
        \lim_{\delta\to0}\lim_{R\to\infty}\lim_{k\to\infty}
        E(\psi_k,\bar J_k^\alpha)=E(\xi^{C,\alpha}),
\]
where $\bar J_k^\alpha$ is the region in $M_k$ corresponding to $J_k^\alpha$.
\end{enumerate}
\end{enumerate}
\end{enumerate}
\end{proposition}

\begin{proof}
(1) follows easily from our assumption. (2) can be shown in the same spirit as in \cite{Zhu}.

For ({\romannumeral1}), it can be argued as in the proof for Proposition 3.1 of \cite{Zhu}.

For ({\romannumeral2}),
Let $\bar{J}_k^{\alpha}$ be the annulus $B(x_{a,k},e^{b_k^{\alpha}})\setminus B(x_{a,k},e^{a_k^{\alpha}})\subset M_k $ corresponding to ${J}_k^{\alpha}$.
Theorem \ref{thm:einstein-compactness} says that  $\bar{J}_k^{\alpha}\subset M_k$ looks like a portion of a flat cone $\mathbb{R}^n/\Gamma$ for large $k$.
So by the relation $a_k^{\alpha}\ll b_k^{\alpha}$, after scaling of the scale
 $\lambda_k^{C,\alpha}\equiv e^{\frac{a_k^{\alpha}-b_k^{\alpha}}{2}}$,
  $\left(\bar{J}_k^{\alpha}, (\lambda_k^{C,\alpha})^{-2}g_k \right)$
  converges to $\mathbb{R}^n/\Gamma$ as $k\rightarrow \infty$, and
   $\psi_k$ converges to $\xi^{C,\alpha}\in\Gamma({\mathbb S(\R^n/\Gamma)})^d$.

\end{proof}

\subsection{Blow-up at the singularity II: multiple ALE bubbles case}
In this subsection, we discuss how to construct the whole bubble tree for the convergence of
{spinor fields} in the case that several ALE bubble spaces emerge at the same orbifold
singularity from the Gromov-Hausdorff convergence of Einstein manifolds.

From the analysis in the previous subsections, we know that three types of bubble
{spinors} appear at the orbifold singularity, namely
\[
        {\xi^A: \R^n\to \Gamma({\mathbb S(\R^n)})^d,
        \qquad \xi^B: M_a\to \Gamma({\mathbb S(M_{a})})^d,
        \qquad \xi^C: \R^n/\Gamma\to \Gamma({\mathbb S(\R^n/\Gamma)})^d,}
\]
where $M_{a}$ represents an ALE bubble space in the tree $\mathbf{Tr_{ALE}}$
(see Theorem \ref{thm:einstein-compactness}).

Bubble of type $\xi^B $ occur in the place where the Riemannian curvatures
 concentrate, and the number of bubble of type $\xi^B$ is
  equal to the number of ALE bubble manifolds (orbifolds) in the tree
  $\mathbf{Tr_{ALE}}$.  Bubble of type $\xi^C$ appear in degenerating
   neck regions such as those discussed in Remark \ref{rem:A-on-B}.
   Bubble of type $\xi^A$ appear over the regions  where there is no
   degeneration of the metrics at its blow-up scale
    (they are on bubble of type $\xi^B$). With these in mind, one can
    easily construct the whole bubble tree by induction. Note that
    the whole process will be terminated in finite steps, since the number
     of non-trivial ALE bubble manifolds (orbifolds) is finite and the
     number of non-trivial bubble of type $\xi^A$ and $\xi^C$
     must be finite by the energy gap theorem and the finite total energy
      assumption.

By the analysis discussed before, to prove Theorem~\ref{thm:main}, we only need to prove the
following theorem.

\begin{theorem}\label{thm:neck-zero-deg}
Let $\bar I_k^\alpha=A_{e^{b_k^\alpha},e^{a_k^{\alpha+1}}}(x_{a,k})$ be the annulus region in
$M_k$ corresponding to $I_k^\alpha$ in Proposition~\ref{prop:bubble-neck}. Then
\[
        \lim_{\delta\to0}\lim_{R\to\infty}\lim_{k\to\infty}
        \int_{\bar I_k^\alpha}|\psi_k|_{g_k}^{\frac{2n}{n-1}}\,dV_{g_k}=0.
\]
\end{theorem}

\vskip1cm

\section{Energy quantization in the degenerating case}\label{section5}

\vskip0.5cm

To prove Theorem~\ref{thm:neck-zero-deg}, we need to investigate the asymptotic {behavior} of the
spinor $\psi_k$ over the neck regions $I_k^\alpha$, which are sub-annuli of the degenerating neck
regions $A_{r_kR,\delta}(x_{a,k})\subset M_k$.

To begin with, we shall firstly recall the following two key results given in the previous sections.
\begin{enumerate}
\item[($\mathfrak{K}_1$)] According to the bubble-neck decomposition in Proposition~\ref{prop:bubble-neck},
$\bar I_k^\alpha\subset A_{r_kR,\delta}(x_{a,k})$, and
\[
        \lim_{\delta\to0}\lim_{R\to\infty}\lim_{k\to\infty}
        \omega(\psi_k,I_k^\alpha)=0.
\]
It says that the energy of $\psi_k$ can be as small as we need on
$A_{\rho,2\rho}(x_{a,k})\subset\bar I_k^\alpha$.
\item[($\mathfrak{K}_2$)] By Theorem~\ref{thm:neck-coordinates}, there exist coordinates $(y)$ on
$A_{r_kR,\delta}(x_{a,k})\subset M_k$ such that
\[
        \|g_{k,jl}(y)-\delta_{jl}\|_{C^{4,\alpha}_{\eta_k}(|y|)}<C.
\]
\end{enumerate}
Notice that the facts $(\mathfrak{K}_1)$ and $(\mathfrak{K}_2)$ in the above ensure that we can apply the
$\varepsilon$-regularity Lemma~\ref{lem:epsilon-reg-fixed} (or its generalizations) on the regions as
$A_{\rho,2\rho}(x_{a,k})$, since if we introduce the coordinates $\tilde y=\rho^{-1}y$ on
$A_{\rho,2\rho}(x_{a,k})$, then
\[
        \|g_{k,jl}(\rho\tilde y)-\delta_{jl}\|_{C^4(B_2\setminus B_1)}<\eta_k(\rho)
\]
for all $\rho\in[r_kR,\delta]$.

\subsection{The Bourguignon-Gauduchon trivialization on degenerating neck regions}
Let \((y_1, \dots, y_n)\)
be the coordinates on $A_{r_kR,\delta}(x_{a,k})\equiv \mathcal{V}\subset M_k$
in Theorem \ref{thm:neck-coordinates}. For simplicity of notations,
 we shall ignore the effect of the fundamental group of $A_{r_kR, \delta}(x_{a,k})$ and omit the lower index $k$ if there is no confusion when we carry out the neck analysis.
For any $y\in\mathcal{V}$, set \(G(y) = (g_{ji}(y))\) where \(g_{ji}(y) = g(\partial_j, \partial_i)\).
Define \(B(y)=\left(b_j^{i}(y)\right)\) as the unique symmetric positive-definite matrix satisfying \(B(y)^2 = G(y)^{-1}\), which induces an isometry
\[
B(y) : (\mathbb{R}^n , g_{\mathbb{R}^n}) \longrightarrow (T_y \mathcal{V}, g(y)).
\]
given by
\[
B(y) : \left( a^1, a^2, \ldots, a^n \right) \to \sum_{j,i} b_j^{i}(y) a^j \partial_i(y).
\]
This further induces a commutative diagram of  \(\mathrm{SO}(n)\)-principal bundles:
\[
\begin{tikzcd}
P_{\mathrm{SO}}({U \subset \mathbb{R}^n}, g_{\mathbb{R}^n}) \arrow[r, "\eta"] & P_{\mathrm{SO}}(\mathcal{V}, g) \\
{U \subset \mathbb{R}^n} \arrow[r, "(y)"] \arrow[u, leftarrow] & \mathcal{V}  \arrow[u, leftarrow]
\end{tikzcd}
\]
where \(\eta(e_1, \dots, e_n) = (Be_1, \dots, Be_n)\)
for an oriented frame $(e_1, \dots, e_n)$ on $U$, $(y)$ is the coordinates chart map in the above.
Since \(\eta\) is \(\mathrm{SO}(n)\)-equivariant, it can be lifted to an isomorphism between the corresponding  $\mathrm{Spin}(n)$-principal bundles:

\[
\begin{tikzcd}
    P_{\mathrm{Spin}}(U, g_{\mathbb{R}^n}) \arrow[r, "\bar{\eta}"] & P_{\mathrm{Spin}}(\mathcal{V}, g) \\
P_{\mathrm{SO}}(U, g_{\mathbb{R}^n}) \arrow[r, "\eta"]\arrow[u, leftarrow] & P_{\mathrm{SO}}(\mathcal{V}, g)\arrow[u, leftarrow] \\
U  \arrow[r, "(y)"] \arrow[u, leftarrow] & \mathcal{V}  \arrow[u, leftarrow]
\end{tikzcd}
\]
This isomorphism induces an isometric isomorphism between the spinor bundles:
\[
\begin{tikzcd}
\mathbb{S}(U, g_{\mathbb{R}^n}):=  P_{\mathrm{Spin}}(U, g_{\mathbb{R}^n})\times_\tau \mathbb{S}_n
 \arrow[r, "\bar{}"] & \mathbb{S}(\mathcal{V}, g):= P_{\mathrm{Spin}}(\mathcal{V}, g)\times_\tau \mathbb{S}_n
\end{tikzcd}
\]
explicitly given by
\[
\psi = [s, \varphi] \longmapsto  [\bar{\eta}(s), \varphi]=:\bar{\psi},
\]
where \([s, \varphi]\) denotes the equivalence class of \((s, \varphi)\) in \(P_{\mathrm{Spin}}(U, g_{\mathbb{R}^n})\times\mathbb{S}_n\).

\begin{proposition}
\label{prop:BG-dege}
    Denote by \(D_{\mathbb{R}^n}\) and \(D_{\mathcal{V}}\) the Dirac operators acting on \(\Gamma\bigl(\mathbb{S}(U, g_{\mathbb{R}^n})\bigr)\) and \(\Gamma\bigl(\mathbb{S}(\mathcal{V}, g)\bigr)\), respectively.
    Then

\[
D_{\mathcal{V}} \bar{\psi} = \overline{D_{\mathbb{R}^n} \psi} + \mathbf{W} \cdot \bar{\psi} + \mathbf{V} \cdot \bar{\psi} + \sum_{i,l} \mathbf{Q}_{il} \overline{\partial_i \cdot \nabla_{\partial_l} \psi},
\]

where
    \[
    \mathbf{W} = \frac{1}{4} \sum_{\substack{m,n,j \\ m \neq n \neq j }} b_m^r (\partial_r b_n^i) (b^{-1})_i^j  e_m \cdot e_n \cdot e_j=O(\frac{\eta_k(|y|)}{|y|}),
    \]
    \[
    \mathbf{V} = \frac{1}{2} \sum_{i,j} \widetilde{\Gamma}_{ij}^i e_j =  O(\frac{\eta_k(|y|)}{|y|}),
    \]
and \(\mathbf{Q}_{ji}=b_j^i - \delta_j^i = O(\eta_k(|y|))\).
Here $\eta_k$ is the weight function in Theorem \ref{thm:neck-coordinates}.
\end{proposition}

\begin{proof}
We follow the arguments in \cite[Proposition 3.2]{AGHM}. Recall that \(B(y)^2 = G(y)^{-1}\), and by Theorem \ref{thm:neck-coordinates}, we have
\begin{equation}\label{almostflat11}
\left\|g_{k,ji}( y)-\delta_{ji}\right\|_{C^{4,\alpha}_{\eta_k(|y|)}}<C,
\end{equation}
so it follows that
\begin{equation}\label{almostorth}
\left\|b_{j}^i( y)-\delta_{ji}\right\|_{C^{4,\alpha}_{\eta_k(|y|)}}<C.
\end{equation}
Indeed, for each $y$ we may decompose $G(y)=Q(y)\Lambda(y) Q^{\bot}(y)$, where $\Lambda(y)=diag(\lambda_1,\cdots,\lambda_n)$ is a diagonal matrix, $Q(y)=(q_j^i)$  is an orthogonal matrix. By Gram-Schmidt orthogonalization process, we know that $\Lambda(y)$ is a diagonal matrix which is close to the identity matrix in the sense that
\begin{equation*}
\left\|\lambda_j(y)-1\right\|_{C^{4,\alpha}_{\eta_k(|y|)}}<C,
\end{equation*}
and
\begin{equation*}
\left\|q_{j}^i( y)-\delta_{ji}\right\|_{C^{4,\alpha}_{\eta_k(|y|)}}<C.
\end{equation*}
By using the knowledge in linear algebra we have that
\[B(y)=Q^{\bot}(y)\Lambda(y)^{-\frac{1}{2}} Q(y).\]
Therefore it is easy to see that \eqref{almostorth} in the above holds.

From \eqref{almostflat11}, it follows by direct computation that $\mathbf{V}  =  O(\frac{\eta_k(|y|)}{|y|})$.  And the estimates for $\mathbf{W}$ and $\mathbf{Q}_{ji}$ are the direct consequences of \eqref{almostorth}.

\end{proof}

Again by using the Bourguignon–Gauduchon trivialization and the transformation formula above, we can now convert the  equation \eqref{eq:varying-system} on the degenerating neck regions $\mathcal{V}$ into a Dirac equation on Euclidean space.
Assume that \(\bar{\psi}_k \in \Gamma(\mathbb{S}(\mathcal{V}))\) satisfies \eqref{eq:varying-system}. Via the Bourguignon–Gauduchon  trivialization, this corresponds to a spinor \(\psi_k \in \Gamma\bigl(\mathbb{S}(U)\bigr)\) on the Euclidean domain \(U\). Applying Proposition~\ref{prop:BG-dege}, we obtain an equivalent equation on \(U\):
\begin{equation}\label{dirac_dege}
    D_{\mathbb{R}^n}\psi_k=H|\psi_k|^{\frac{2}{n-1}}\psi_k +h_k
\end{equation}
where \(h_k\) satisfies
\[
|h_k| \leq C\,\left(O\left(\frac{\eta_k(|y|)}{|y|}\right)|\psi|+ O(\eta_k(|y|))|\nabla \psi| \right).
\]

Next, we employ the conformal invariance of the Dirac operator to transform equation \eqref{dirac_dege} into a Dirac equation on a cylinder.
For \((t, \theta)\in [-\ln \delta, -\ln (\lambda_k R)]\times S^{n-1}\), let \((r, \theta)\) be polar coordinates on $\mathbb{R}^n$ and define
\[
f: \Sigma := [-\ln \delta, -\ln (\lambda_k R)] \times S^{n-1} \longrightarrow U,\quad f(t, \theta) = (e^{-t}, \theta),
\]
where \([-\ln \delta, -\ln (\lambda_k R)] \times S^{n-1}\) is endowed with the product metric \(g_{\mathrm{cyl}} = dt^2 + d\theta^2\). The map \(f\) is a conformal diffeomorphism and it yields an isomorphism from \((T\Sigma, e^{-2t}g_{\mathrm{cyl}})\) to \((TU, g_{\mathbb{R}^n})\), which may be viewed as an isomorphism of \(\mathrm{SO}(n)\)-principal bundles. It lifts to an isomorphism of \(\mathrm{Spin}(n)\)-principal bundles and induces an isometric isomorphism

\[
F:\mathbb{S}(\Sigma, g_{\mathrm{cyl}})\longrightarrow \mathbb{S}(U, g_{\mathbb{R}^n}).
\]
As before we obtain the following commutative diagram:

\[
\begin{tikzcd}
    \mathbb{S}(\Sigma, g_{\mathrm{cyl}})\arrow[r, "F"]&
    \mathbb{S}(U, g_{\mathbb{R}^n})\arrow[r, "\bar{\eta}"]& \mathbb{S}(\mathcal{V}, g) \\
\Sigma \arrow[r, "f"] \arrow[u, leftarrow] &
U \arrow[r, "\exp_q"] \arrow[u, leftarrow] & \mathcal{V}  \arrow[u, leftarrow]
\end{tikzcd}
\]
Moreover, with respect to the metrics \(g_{\mathbb{R}^n}\) and \(g_{\mathrm{cyl}}\), the Dirac operators satisfy
\begin{equation}
\label{eq: confdege}
    D_{\Sigma}\,F^{-1}\!\Bigl(e^{-\frac{n-1}{2}t}\psi\Bigr)
    = e^{-\frac{n+1}{2}t}
    F^{-1}\!\Bigl(D_{{\mathbb{R}^n}}\psi\Bigr).
\end{equation}
Hence, if \(\psi_k\) is a solution of \eqref{dirac_dege}, then \(\varphi_k = F^{-1}\!\bigl(e^{-\frac{n-1}{2}t}\psi_k\bigr)\) satisfies
\begin{equation}
\label{eq: cylidege}
    D_{\Sigma}\varphi_k = |\varphi_k|_{g_{\mathrm{cyl}}}^{\frac{2}{n-1}} \varphi_k + \hat{h}_k,
\end{equation}
where \(\hat{h}_k= e^{-\frac{n+1}{2}t}F^{-1}h_k\), and

\[
\int_{\Sigma}|\varphi_k|_{g_{\mathrm{cyl}}}^{\frac{2n}{n-1}}\,dV_{g_{\mathrm{cyl}}}
= \int_{U}|\psi_k|_{g_{\mathbb{R}^n}}^{\frac{2n}{n-1}}\,dV_{g_{\mathbb{R}^n}} .
\]

\subsection{Elliptic estimates on degenerating neck regions}
\begin{lemma}\label{lem:neck-eps-reg}
Let $p>1$ and let $\psi:B_\delta\setminus B_{r_kR}\to{\mathbb S(\R^n)}$ be a solution of
\[
        D_{\R^n}\psi=|\psi|^{2/(n-1)}\psi+a(y)\psi+b(y)\nabla\psi,
\]
where $a(y)$ and $b(y)$ are smooth functions satisfying
$a(y)=O(\eta_k(|y|)/|y|)$ and $b(y)=O(\eta_k(|y|))$. Then for any
$r_kR<2r<5r/2<\delta$ there exist constants $\varepsilon>0$ and $C>0$ such that if
\[
        \int_{B_{5r/2}\setminus B_{r/2}}|\psi|^{2n/(n-1)}<\varepsilon,
\]
it holds that
\begin{equation}\label{eq:neck-W1p}
        \|\psi\|_{W^{1,p}(B_{2r}\setminus B_r)}
        \le Cr^{\frac{2n-p(n+1)}{2p}}
        \|\psi\|_{L^{2n/(n-1)}(B_{5r/2}\setminus B_{r/2})}.
\end{equation}
\end{lemma}

\begin{proof}
By using a trick of scaling as in Lemma \ref{lem:epsilon-reg-fixed}, we may assume that it holds on $B_{5/2}\setminus B_{1/2}$ that
\[
        D_{\R^n}\widetilde\psi=|\widetilde\psi|^{2/(n-1)}\widetilde\psi
        +\widetilde a(\widetilde y)\widetilde\psi+\widetilde b(\widetilde y)\nabla\widetilde\psi,
\]
where $r\widetilde y=y$, $\widetilde a$ and $\widetilde b$ are smooth functions satisfying
$\widetilde a(\widetilde y)=O(\eta_k(r))$ and $\widetilde b(\widetilde y)=O(\eta_k(r))$. By the conformal invariance we have
\[
        \int_{B_{5/2}\setminus B_{1/2}}|\widetilde\psi|^{2n/(n-1)}<\varepsilon.
\]
Then we can apply the argument in Lemma~\ref{lem:epsilon-reg-fixed} to show that
\begin{equation}\label{eq:scaled-neck-W1p}
        \|\widetilde\psi\|_{W^{1,p}(B_2\setminus B_1)}
        \le C\|\widetilde\psi\|_{L^{2n/(n-1)}(B_{5/2}\setminus B_{1/2})}.
\end{equation}
It follows that \eqref{eq:neck-W1p} holds.

For the reader's convenience, we give the details for the proof of \eqref{eq:scaled-neck-W1p} here. For simplicity of notations, we write $\psi,a,b$ for
$\widetilde\psi,\widetilde a,\widetilde b$. Choose a cutoff function
$\zeta\in C_c^\infty(B_{5/2}\setminus B_{1/2})$ with $\zeta\equiv1$ on
$B_2\setminus B_1$, $0\le\zeta\le1$, and $|\nabla\zeta|\le C$. Then
\[
        D_{\R^n}(\zeta\psi)=\zeta D_{\R^n}\psi+\nabla\zeta\cdot\psi
        =\zeta|\psi|^{2/(n-1)}\psi+a\zeta\psi+b\zeta\nabla\psi+\nabla\zeta\cdot\psi.
\]
By the ellipticity of the Dirac operator \(D_{\R^n}\),
\[
        \|\zeta\psi\|_{W^{1,p}(B_{5/2}\setminus B_{1/2})}
        \le C\left(\|\zeta\psi\|_{L^p(B_{5/2}\setminus B_{1/2})}
        +\|D_{\R^n}(\zeta\psi)\|_{L^p(B_{5/2}\setminus B_{1/2})}\right),
\]
which yields
\[
\begin{aligned}
\|\zeta\psi\|_{W^{1,p}(B_{5/2}\setminus B_{1/2})}
\le C\big(&\|\zeta\psi\|_{L^p(B_{5/2}\setminus B_{1/2})}
+\|\zeta|\psi|^{2/(n-1)}\psi\|_{L^p(B_{5/2}\setminus B_{1/2})}\\
&+\|a\zeta\psi\|_{L^p(B_{5/2}\setminus B_{1/2})}
+\|b\zeta\nabla\psi\|_{L^p(B_{5/2}\setminus B_{1/2})}
+\|\nabla\zeta\cdot\psi\|_{L^p(B_{5/2}\setminus B_{1/2})}\big).
\end{aligned}
\]
Notice that $a,b=O(\eta_k(r))$. For $1<p<n$, H\"older's inequality and the Sobolev inequality give
\[
\begin{aligned}
\|\zeta\psi\|_{W^{1,p}(B_{5/2}\setminus B_{1/2})}
&\le C\left(\|\psi\|_{L^p(B_{5/2}\setminus B_{1/2})}
+\|\zeta|\psi|^{2/(n-1)}\psi\|_{L^p(B_{5/2}\setminus B_{1/2})}\right)\\
&\quad +CO(\eta_k(r))\|\zeta\psi\|_{W^{1,p}(B_{5/2}\setminus B_{1/2})}\\
&\le C\left(\|\psi\|_{L^p(B_{5/2}\setminus B_{1/2})}
+\|\zeta\psi\|_{L^{\frac{np}{n-p}}(B_{5/2}\setminus B_{1/2})}
\|\psi\|_{L^{\frac{2n}{n-1}}(B_{5/2}\setminus B_{1/2})}^{\frac2{n-1}}\right)\\
&\quad +CO(\eta_k(r))\|\zeta\psi\|_{W^{1,p}(B_{5/2}\setminus B_{1/2})}\\
&\le C\left(\|\psi\|_{L^p(B_{5/2}\setminus B_{1/2})}
+\|\zeta\psi\|_{W^{1,p}(B_{5/2}\setminus B_{1/2})}
\|\psi\|_{L^{\frac{2n}{n-1}}(B_{5/2}\setminus B_{1/2})}^{\frac2{n-1}}\right)\\
&\quad +CO(\eta_k(r))\|\zeta\psi\|_{W^{1,p}(B_{5/2}\setminus B_{1/2})}.
\end{aligned}
\]
Since
\[
        \int_{B_{5/2}\setminus B_{1/2}}|\psi|^{2n/(n-1)}<\varepsilon,
\]
for sufficiently small $\varepsilon$ and $O(\eta_k(r))$ (recall that
$\eta_k(r)=O((r_k/r)^{\varepsilon_5}+(r/r_\infty)^{\varepsilon_5})$), there holds
\begin{equation}\label{eq:neck-absorb-restored}
        \|\zeta\psi\|_{L^{\frac{np}{n-p}}(B_{5/2}\setminus B_{1/2})}
        \le \|\zeta\psi\|_{W^{1,p}(B_{5/2}\setminus B_{1/2})}
        \le C\|\psi\|_{L^p(B_{5/2}\setminus B_{1/2})}.
\end{equation}
{If $\frac{2n}{n-1}<n$,} set $p_0=\frac{2n}{n-1}$ and define $p_{m+1}=\frac{np_m}{n-p_m}$. {If $\frac{2n}{n-1}\ge n$, start instead from any exponent $p_0<n$ below the critical exponent and then perform one more bootstrap step.} There exists $m$ such that
$p_m>n$. Iteratively applying \eqref{eq:neck-absorb-restored}, we obtain for any $1<p<n$,
\begin{equation}\label{eq:neck-lp-iteration-restored}
        \|\psi\|_{L^p(B_2\setminus B_1)}
        \le C\|\psi\|_{L^{2n/(n-1)}(B_{5/2}\setminus B_{1/2})}.
\end{equation}
Using \eqref{eq:neck-absorb-restored} again shows that \eqref{eq:neck-lp-iteration-restored} holds for any $p>1$. Moreover,
\[
\begin{aligned}
\int_{B_{5/2}\setminus B_{1/2}} |D_{\R^n}(\zeta\psi)|^p
&=\int_{B_{5/2}\setminus B_{1/2}} |\zeta D_{\R^n}\psi+\nabla\zeta\cdot\psi|^p\\
&\le C\int_{B_{5/2}\setminus B_{1/2}}
\left(|\psi|^{\frac{p(n+1)}{n-1}}+|\psi|^p+O(\eta_k(r))|\nabla(\zeta\psi)|^p\right),
\end{aligned}
\]
which, together with the ellipticity of the Dirac operator $D_{\R^n}$ and
\eqref{eq:neck-lp-iteration-restored}, implies
\[
        \|\psi\|_{W^{1,p}(B_2\setminus B_1)}
        \le C\|\psi\|_{L^{2n/(n-1)}(B_{5/2}\setminus B_{1/2})}.
\]
\end{proof}
Repeating the argument in Lemma \ref{lem:linear-eps}, we obtain the following result.
\begin{lemma}\label{lem:neck-linear}
Let $p>1$ and let $\psi:B_\delta\setminus B_{r_kR}\to{\mathbb S(\R^n)}$ be a solution of
\[
        D_{\R^n}\psi=|\phi|^{2/(n-1)}\psi+a(y)\psi+b(y)\nabla\psi,
\]
where $a(y)$ and $b(y)$ are smooth functions satisfying $a(y)=O(\eta_k(|y|)/|y|)$ and
$b(y)=O(\eta_k(|y|))$, and
\[
        \|\psi\|_{L^{2n/(n-1)}}+\||y|^{\frac{n-1}{2}}\psi\|_{L^\infty}\le C.
\]
Then
\[
        \|\psi\|_{W^{1,2}(B_{2r}\setminus B_r)}
        \le Cr^{-1}\|\psi\|_{L^2(B_{5r/2}\setminus B_{r/2})},
\]
where the constant $C>0$ depends only on $n$, $p$, and the bound for
$\|\psi\|_{L^{2n/(n-1)}}+\||y|^{\frac{n-1}{2}}\psi\|_{L^\infty}$.
\end{lemma}

\subsection{Three circles arguments on the degenerating neck regions}
Suppose
\[
        b_k^\alpha=-l_{k,b}^\alpha L,
        \qquad
        a_k^{\alpha+1}=-l_{k,a}^\alpha L
\]
for some universal constant $L>0$. Set
\[
        A_l=A_{k;e^{-lL},e^{-(l-1)L}}
        =B(x_{a,k},e^{-(l-1)L})\setminus B(x_{a,k},e^{-lL}),
        \qquad l_{k,a}^\alpha\le l<l_{k,b}^\alpha.
\]
By Theorem~\ref{thm:neck-coordinates}, for some $\widetilde L$,
\[
        A_l\subset \widetilde A_{l-1}\cup\widetilde A_l\cup\widetilde A_{l+1},
\]
where
\[
        \widetilde A_l=\{q\in M_k;\ e^{-l\widetilde L}\le |y(q)|\le e^{-(l-1)\widetilde L}\}.
\]
Let $\Sigma_l$ be the corresponding region of $\widetilde A_l$ in cylinder coordinates $(t,\theta)$ in the above. And let $\phi$ be the corresponding spinor on the cylinder $\Sigma$ to the spinor $\psi$ in the coordinates $(y)$.

\begin{lemma}\label{lem:neck-three-circle}
Let $\phi:\Sigma_{l-1}\cup\Sigma_l\cup\Sigma_{l+1}\to{\mathbb S(\Sigma)}$ satisfy
\[
        D_\Sigma\phi=|\phi|_{g_{\cyl}}^{2/(n-1)}\phi+\hat h,
\]
where $\hat h$ satisfies
\[
        |\hat h|\le C\eta_k(e^{-t})(|\phi|_{g_{\cyl}}+|\nabla_{\cyl}\phi|_{g_{\cyl}}).
\]
Then there exist constants $\varepsilon,\tau>0$ small enough such that if
$\|\phi\|_{L^{2n/(n-1)}(\Sigma_l)}<\varepsilon$ and
\[
        \max_{l-1,l,l+1}\|\hat h\|_{L^p(\Sigma_l)}\le \tau\|\phi\|_{L^2(\Sigma_l)},
\]
the following hold:
\begin{enumerate}[label=(\alph*)]
\item $\|\phi\|_{L^2(\Sigma_{l+1})}\le e^{-L/2}\|\phi\|_{L^2(\Sigma_l)}$ implies
$\|\phi\|_{L^2(\Sigma_l)}\le e^{-L/2}\|\phi\|_{L^2(\Sigma_{l-1})}$;
\item $\|\phi\|_{L^2(\Sigma_{l-1})}\le e^{-L/2}\|\phi\|_{L^2(\Sigma_l)}$ implies
$\|\phi\|_{L^2(\Sigma_l)}\le e^{-L/2}\|\phi\|_{L^2(\Sigma_{l+1})}$;
\item Either $\|\phi\|_{L^2(\Sigma_l)}\le e^{-L/2}\|\phi\|_{L^2(\Sigma_{l+1})}$ or
$\|\phi\|_{L^2(\Sigma_l)}\le e^{-L/2}\|\phi\|_{L^2(\Sigma_{l-1})}$.
\end{enumerate}
\end{lemma}

The lemma can be proved by the same argument as in the proof of Lemma~\ref{lem:three-circle-perturbed}, except that we replace Lemma~\ref{lem:epsilon-reg-fixed} and Lemma~\ref{lem:linear-eps} with Lemma~\ref{lem:neck-eps-reg} and Lemma~\ref{lem:neck-linear}.

Noticing that
\[
        \eta_k(r)=O\left(\left(\frac{r_k}{r}\right)^{\varepsilon_5}+
        \left(\frac{r}{r_\infty}\right)^{\varepsilon_5}\right),
\]
the following lemma is a direct corollary of Lemma~\ref{lem:neck-three-circle}.

\begin{lemma}\label{lem:neck-decay}
Let $\phi:\Sigma\to{\mathbb S(\Sigma)}$ satisfy
\[
        D_\Sigma\phi=|\phi|_{g_{\cyl}}^{2/(n-1)}\phi+\hat h,
\]
where $\hat h$ satisfies
\[
        |\hat h|\le C\eta_k(e^{-t})(|\phi|_{g_{\cyl}}+|\nabla_{\cyl}\phi|_{g_{\cyl}}),
\]
and assume that $k$, $R$ and $1/\delta$ are sufficiently big and for every
$l_{k,a}^\alpha\le l<l_{k,b}^\alpha$,
$\|\phi\|_{L^{2n/(n-1)}(\Sigma_l)}<\varepsilon$. Then the following estimate holds:
\[
        \|\phi\|_{L^2(\Sigma_l)}^2
        \le C\varepsilon^2\left(e^{-\vartheta(l-l_{k,a}^\alpha)\widetilde L}
        +e^{-\vartheta(l_{k,b}^\alpha-l)\widetilde L}\right).
\]
Here $\vartheta=\min(1,2\varepsilon_5)$.
\end{lemma}

Finally we give the proof of Theorem~\ref{thm:main} by an argument similar to Theorem~\ref{thm:fixed}.

\begin{proof}[Proof of Theorem~\ref{thm:main}]
The proof of Theorem~\ref{thm:main} is reduced to proving
\begin{equation}\label{eq:deg-neck-zero}
        \lim_{\delta\to0}\lim_{R\to\infty}\lim_{k\to\infty}
        \int_{\bar I_k^\alpha}|\psi_k|^{\frac{2n}{n-1}}\,dV_{g_k}=0.
\end{equation}
For any $\varepsilon>0$, we know when $j$, $R$ and $1/\delta$ are big enough,
\[
        \|\phi_k\|_{L^{2n/(n-1)}(\Sigma_l)}<\varepsilon.
\]
By Lemma~\ref{lem:neck-eps-reg}, $\|\phi_k\|_{L^\infty(\Sigma)}\le C\varepsilon$. Let
$\Sigma_k^\alpha$ be the region in the cylinder coordinates corresponding to $\bar I_k^\alpha$.
Then, using Lemma~\ref{lem:neck-decay},
\[
\begin{aligned}
\int_{\bar I_k^\alpha}|\psi_k|^{\frac{2n}{n-1}}\,dV_{g_k}
&\leq C\int_{\Sigma_k^\alpha}|\phi_k|^{\frac{2n}{n-1}}\,dV_{g_{\cyl}}\\
&\le C\int_{\Sigma_k^\alpha}|\phi_k|^2\,dV_{g_{\cyl}}\\
&\le C\varepsilon^2
\sum_{l=l_{k,a}^\alpha}^{l_{k,b}^\alpha}
\left(e^{-\vartheta(l-l_{k,a}^\alpha)\widetilde L}
+e^{-\vartheta(l_{k,b}^\alpha-l)\widetilde L}\right).
\end{aligned}
\]
Then \eqref{eq:deg-neck-zero} follows immediately. Combining this with the bubble-neck decomposition gives the energy identity in Theorem~\ref{thm:main}.
\end{proof}

\end{document}